\documentclass{article}

\usepackage[a4paper, total={6in, 8in}]{geometry}

\usepackage{graphicx}
\usepackage{amssymb}
\usepackage{amsmath}
\usepackage{amsthm}
\usepackage{bbm}
\usepackage{subcaption}
\usepackage{comment}

\newtheorem{theorem}{Theorem}[section]
\newtheorem{lemma}[theorem]{Lemma}
\newtheorem{proposition}[theorem]{Proposition}

\newtheorem{example}[theorem]{Example}
\newtheorem{definition}[theorem]{Definition}
\newtheorem{observation}[theorem]{Observation}

\newcommand{\ppspace}{\statespace}

\newcommand{\zeromeasure}{\bar{o}}
\newcommand{\supmark}[1]{\mathsf{m}(#1)}
\newcommand{\dist}[1]{\mathsf{P}_{#1}}
\newcommand{\Msetl}[1]{\underline{\mathcal{M}}^{#1}}
\newcommand{\DLR}{\textbf{DLR}}
\newcommand{\Msetsigma}{\mathfrak{M}}
\newcommand{\IntProGammaf}[2]{\left\langle #1, #2 \right\rangle}

\newcommand{\barpn}{\bar{\mathsf{P}}_n}
\newcommand{\barp}{\bar{\mathsf{P}}}
\newcommand{\hatpn}{\hat{\mathsf{P}}_n}

\newcommand{\tildepn}{\tilde{\mathsf{P}}_n}

\newcommand{\eps}{\varepsilon}

\newcommand{\gd}{\mathcal{G}_d}

\newcommand{\unitsphere}{\mathbb{S}^{d-1}}
\newcommand{\hem}{\mathbb{S}^{d-1}_+}
\newcommand{\hemdva}{\mathbb{S}^{1}_+}

\newcommand{\dotproduct}[2]{\left\langle #1, #2 \right\rangle}

\newcommand{\C}{C}
\newcommand{\B}{B}
\newcommand{\A}{A}
\newcommand{\R}{\mathbb{R}}
\newcommand{\rtwo}{\mathbb{R}^2}
\newcommand{\rthree}{\mathbb{R}^3}
\newcommand{\rd}{\mathbb{R}^d}
\newcommand{\N}{\mathbb{N}}

\newcommand{\ztwo}{\mathbb{Z}^2}

\newcommand{\boreld}{\mathcal{B}^d}

\newcommand{\boreldom}{\mathcal{B}^d_b}
\newcommand{\boreldvaom}{\mathcal{B}^2_b}

\newcommand{\drm}{\mathrm{d}}
\newcommand\sumne{\mathop{\sum\nolimits^{\not =}}\limits} 

\newcommand{\ind}{\textbf{1}}

\newcommand{\buno}{w.\,l.\,o.\,g.}
\newcommand{\wrt}{w.\,r.\,t.}

\newcommand{\markspace}{\mathcal{S}}
\newcommand{\statespace}{\rd \times \markspace}
\newcommand{\hmoment}{\mathcal{H}_m}
\newcommand{\hs}{\mathcal{H}_s}
\newcommand{\hl}{\mathcal{H}_l}
\newcommand{\hr}{\mathcal{H}_r}

\newcommand{\Mset}{\mathcal{M}}
\newcommand{\Msett}{\mathcal{M}^t}
\newcommand{\Msettemp}{\mathcal{M}^{\textit{temp}}}
\newcommand{\Msetf}{\mathcal{M}_f}
\newcommand{\MsetLambda}{{\mathcal{M}_{\Lambda}}}
\newcommand{\ProbMset}{\mathcal{P}(\mathcal{M})}
\newcommand{\poislambdaz}{\pi_\Lambda^z}
\newcommand{\poislambdanz}{\pi_{\Lambda_n}^z}
\newcommand{\markdist}{\mathsf{Q}}
\newcommand{\norm}[1]{\|#1\|}
\newcommand{\ms}[2]{#1_{#2}}
\newcommand{\abs}[1]{\left|{#1}\right|}
\newcommand{\goto}{\rightarrow}

\title{A Note on the Existence of Gibbs Marked Point Processes with Applications in Stochastic Geometry}

\author{Martina Petr{\'a}kov{\'a},\\ Faculty of Mathematics and Physics, Department of Probability \\ and Mathematical Statistics, Charles University, Sokolovská 83, \\ Prague 8, 18675, Czech Republic}

\date{March 20, 2023}

\begin{document}
\maketitle

\large \textbf{Abstract} 

\smallskip

\normalsize
This paper generalizes a recent existence result for infinite-volume marked Gibbs point processes. We try to use the existence theorem for two models from stochastic geometry. First, we show the existence of Gibbs facet processes in~$\mathbb{R}^d$ with repulsive interactions. We also prove that the~finite-volume Gibbs facet processes with attractive interactions need not exist. Afterwards, we study Gibbs--Laguerre tessellations of $\mathbb{R}^2$. The mentioned existence result cannot be used, since one of its assumptions is not satisfied for tessellations, but we are able to show the existence of an infinite-volume Gibbs--Laguerre process with a particular energy function, under the assumption that we almost surely see a~point.

\smallskip



\section{Introduction}

Gibbs point processes present a broad family of models that considers various possibilities of interactions between points. The effect of these interactions is explained through the notion of an energy function, with states possessing lower energy being more probable than states possessing higher energy. This convention stems from the physical interpretation, as the notion of Gibbs processes was first introduced in statistical mechanics; see \cite{bo:RuelleUvod} for the standard reference.
Among others, \cite{bo:MollerWaagTheory2003} and \cite{ar:DereudreIntro} provide a general introduction to the topic of Gibbs point processes in the context of spatial modelling.

Gibbs point processes in a bounded window are defined using a density with respect to the distribution of a Poisson point process;
however, the situation gets much more complicated once we start to consider processes in the whole~$\rd$. As one can no longer use the approach with a~density with respect to a reference process, one can no longer define the distribution of an infinite-volume Gibbs process (called the infinite-volume Gibbs measure) explicitly. Instead, the \textbf{DLR} equations (see \cite{bo:RuelleUvod}) are used, which prescribe the distribution of the process inside a bounded window conditionally on a fixed configuration outside of this window.

The standard approach to obtain an infinite-volume Gibbs measure is based on the topology of local convergence and the result from \cite{ar:GeorgiiZessin} for level sets of a specific entropy.
One of the standard assumptions for the energy function is the finite-range assumption, which ensures that the range of interactions is uniformly bounded. It was proved in \cite{ar:DereudreQuermass} that the quermass-interaction process with unbounded grains (i.\,e., unbounded interactions) exists. Using this paper as an inspiration, an existence result for marked Gibbs point processes with unbounded interaction was proved in \cite{ar:RoellyZass}.

In the present paper, we address the assumptions of the existence theorem from \cite{ar:RoellyZass}  and present a modified version of the range assumption. Afterwards, we present two applications on models from stochastic geometry.

The first one is the Gibbs facet process in $\rd$ (see \cite{ar:VeceraBenesFasety}).
Here the energy is a~function of the intersections of tuples of facets. We prove that the repulsive model satisfies assumptions of the existence theorem, and therefore the infinite-volume Gibbs facet process exists in this case. On the other hand, we find a counterexample in $\rtwo$ for the case with attractive interactions and we extend it to prove that the finite-volume Gibbs facet processes with attractive interactions do not exist in $\rtwo$.

The second application deals with a model for a random tessellation of $\rtwo$.
We consider the Laguerre tessellation $L(\gamma)$ (see \cite{ar:LautensackZuyev}), which partitions $\rtwo$ according to the  power distance w.r.t.\ at most countable set of generators $\gamma \subset \rtwo \times (0,\infty)$.
We are interested in the situation where the random set of generators is a marked Gibbs point process with the energy function depending on the geometric properties of the tessellation.

Gibbs point processes with geometry-dependent interactions (which include random tessellations) were considered in \cite{ar:DDG2012}. Using the concept of hypergraph structure, an existence result was derived for the unmarked case. It was remarked that the~same existence result would extend to the marked case, and based on this, the existence of an infinite-volume Gibbs measure for several models of Gibbs--Laguerre tessellations of $\rthree$ was derived in~\cite{ar:JS2020} under the assumption of bounded marks.

In the case of marks not being uniformly bounded, while the range assumption from~\cite{ar:RoellyZass} turned out to be more restricting than initially expected, we
noticed that one can still use some of the results from that paper to
support the proof of the existence of the Gibbs--Laguerre tessellation. After a careful analysis of the behaviour of the Laguerre diagram, we
considered as an example the model with energy given by the number of
vertices in the tessellation, where we were able to prove a new existence theorem under the condition that we almost surely see a point.

\section{Basic notation and definitions}\label{sec:Intro}
In this paper, we study simple marked point processes. Our \textit{state space} will be in the product form $\rd \times \markspace$, where $d\geq 2$ and the  \textit{mark space} $(\markspace, \norm{\cdot})$ is a normed space. Each point $(x,m) \in \rd \times \markspace$ consists of the location $x \in \rd$ and the mark $m \in \markspace$. Let $\mathcal{B}(\statespace)$ denote the Borel $\sigma-$algebra on $\statespace$
and let $\boreld$ and $\boreldom$  denote the Borel $\sigma-$algebra and the set of all bounded Borel subsets of $\rd$, respectively.

We denote by $\Mset$ the set of all \textit{simple counting locally finite} Borel measures on $\statespace$ such that their projections $\gamma'(\cdot) = \gamma (\cdot \times \markspace)$ on $\rd$ are also simple counting locally finite Borel measures.
Each $\gamma \in \Mset$ (often referred to as \textit{configuration}) can be represented as $$\gamma = \sum_{i=1}^{N}\delta_{(x_i,m_i)},$$ where $\delta_{(\cdot)}$ denotes the \textit{Dirac measure}, $(x_i,m_i) \in \rd \times \markspace$, where $x_i$ are pairwise different points and $N \in \N \cup \{0,\infty\}$. Therefore, we can identify $\gamma$ with its support (the \textit{zero measure} $\zeromeasure $ is identified with $\emptyset$),
$\gamma \equiv \textbf{supp }\gamma = \{(x_1,m_1),(x_2,m_2), \dots \} \subset \ppspace.$
As is usual, we will sometimes regard $\gamma \in \Mset$ as a (locally finite) subset of $\statespace$ instead of a~simple counting locally finite measure for the sake of simple notation.

\textit{Simple marked point process} is a random element in the space $(\Mset, \Msetsigma)$. Here $\Msetsigma$ is the usual $\sigma$-algebra on $\Mset$ defined as the smallest $\sigma$-algebra such that the projections $p_B: \Mset \goto \R$, where $p_B(\nu) = \nu(B)$, are measurable $\forall B \in \mathcal{B}(\ppspace)$. The \textit{distribution} of a given point process is a probability measure on $(\Mset, \Msetsigma)$.

Take $\gamma,\,\xi \in \Mset$, $z \in \rd$, $\Lambda \in \boreld$ and denote by $\gamma_\Lambda = \sum_{i: x_i \in \Lambda}\delta_{(x_i,m_i)}$ the restriction of $\gamma$ to $\Lambda \times \markspace$ and
		by $\abs{\gamma} = \gamma(\ppspace)$ the number of points of $\gamma$.
The sum of measures $\gamma$ and $\xi$ is denoted by
	$\gamma \, \xi = \sum_{(x,m) \in \gamma} \delta_{(x,m)} + \sum_{(y,n) \in \xi}\delta_{(y,n)}$ and the supremum of norms of all marks in $\gamma$ is denoted by
	$\supmark{\gamma} = \underset{(x,m) \in \gamma}{\mathrm{sup}} \norm{m}$.

	Let $f: \ppspace \goto \R$ be a measurable, $\gamma$-integrable function. We write
	$$\IntProGammaf{\gamma}{f} = \int f(\boldsymbol{x}) \, \gamma(\drm \boldsymbol{x}) = \sum_{\boldsymbol{x} \in \gamma} f(\boldsymbol{x}).$$
We define the set of configurations with points in $\Lambda \times \markspace$ as $\MsetLambda = \{\gamma \in \Mset: \gamma = \ms{\gamma}{\Lambda}\}$. The set of all finite configurations is denoted by $\Msetf = \{\gamma \in \Mset: \left|\gamma \right| < \infty \}$ and for $a > 0$ define $\Mset_a = \{\gamma \in \Mset: \supmark{\gamma} \leq a\}$.

Take $x \in \rd$ and $r > 0$, then the open ball with centre  $x$ and radius $r$  is denoted by $U(x,r)$ and the closed ball with centre $x$ and radius $r$ by $B(x,r)$. The complement of a~ set $A \subset \rd$ will be denoted by $A^c$, the  \textit{interior} of $A$ by $\text{int}(A)$, the \textit{closure} of $A$ by $\text{clo}(A)$  and $\text{bd}(A) = \text{clo}(A) \setminus \text{int}(A)$ denotes the \textit{boundary} of $A$.

 Let $\Lambda \in \boreldom$. A function $F:\Mset \goto \R$ is called \textit{local} (or \textit{$\Lambda$-local}), if it satisfies $F(\gamma)=F(\gamma_\Lambda)$ for all $\gamma \in \Mset$.

\subsection{Tempered configurations and Gibbs measures} \label{subsec:TCaGM}
From now on, we fix $\delta > 0$. Before we dive into the theory of Gibbs processes, we define the set of tempered configurations (for reference, see Section 2.2 in \cite{ar:RoellyZass}).
The importance of this definition lies in the fact that the infinite-volume Gibbs measure is concentrated on the set of tempered configurations.

Take $t \in \N$ and set $ \Msett = \left \lbrace \gamma \in \Mset:  \left\langle \ms{\gamma}{U(0,\l)}, (1+\norm{m}^{d+\delta}) \right\rangle \leq t \cdot\l^d \text{ holds }\forall \l \in \N \right \rbrace$. Then $\Msettemp = \bigcup_{t \in \N} \Msett$ is called the set of \textit{tempered configurations}. These configurations have the following property (for proof, see Lemma 2 in \cite{ar:RoellyZass}). For $t \in \N$ there exists $l(t)$ such that $\forall l \geq l(t)$ and $\forall \gamma \in \Msett$ the following implication holds
	\begin{equation} \label{lemma:vlasttemp}
		(x,m) \in \ms{\gamma}{U(0,2\l+1)^c} \implies B(x,\norm{m}) \cap U(0,\l) = \emptyset.
	\end{equation}

This property inspires the following definition of an increasing sequence of subsets of $\Msettemp$. Take $\l \in \N$ and define
	\begin{equation*}
		\Msetl{l} =  \lbrace \gamma \in \Msettemp: B(x,\norm{m}) \cap U(0,k) = \emptyset ,\, \forall (x,m) \in \gamma_{U(0,2k+1)^c},\, \forall k \in \N, \, k \geq l \rbrace.
	\end{equation*}

We can see from (\ref{lemma:vlasttemp}) that $\Msett \subset \Msetl{\lceil l(t) \rceil}$, $\forall t \in \N$, and consequently $\Msettemp = \bigcup_{\l \in \N} \Msetl{\l}.$ For simplicity, we will write $\Msetl{l(t)}$ instead of $\Msetl{\lceil l(t) \rceil}$ in the following text.

 The focus of this work is the family of \textit{Gibbs point processes}. In particular, we will work with the distributions of these processes, which are called \textit{Gibbs measures}. Choose a~\textit{reference mark distribution} $\markdist$ on the mark space $(\markspace, \norm{\cdot})$ and take $\Lambda \in \boreldom$ and $z > 0$. As a \textit{reference distribution} take $\poislambdaz$, the distribution of the marked Poisson point process in $\statespace$ with intensity measure $z \lambda_\Lambda(\drm x) \bigotimes \markdist(\drm m)$, where $\lambda_\Lambda(\drm x)$ is the restriction of the Lebesgue measure $\lambda$ on $\Lambda$ and $ \bigotimes $ denotes the standard product of measures.

	\textit{An energy function} is a mapping $H: \Msetf \goto \R\cup\{+\infty\}$ which is  measurable, translation invariant and satisfies $H(\zeromeasure) = 0$. Take $\Lambda \in \boreldom$ and $z > 0$. We define \textit{the finite-volume Gibbs measure} in $\Lambda$ with energy function~$H$ and activity $z$ as
	\begin{equation}\label{def:FVGM}
		\dist{\Lambda}(\drm \gamma) = \frac{1}{Z_\Lambda} \cdot e^{-H(\gamma_\Lambda)}\, \poislambdaz(\drm \gamma),
	\end{equation}
where $Z_\Lambda = \int e^{-H(\gamma_\Lambda)} \poislambdaz(\drm \gamma)$ is the normalizing constant called the \textit{partition function}.

Clearly, for the finite-volume Gibbs measure to be well defined, we need $0 < Z_\Lambda < \infty$. This will be satisfied under our assumptions on the energy function $H$ (see Section \ref{sec:GMIntro}).
\begin{example}[Example 2 in \cite{ar:RoellyZass}]\label{Ex:EnergyFuncPair}
	Let $\phi: \rd \times \rd \goto \R$ be a non-negative, translation invariant, measurable function, called the \textit{pair potential}, and consider
	\begin{equation*}
			H_1(\gamma) = \sumne_{(x,m),\,(y,n) \in \gamma} \phi(x,y) \cdot \ind \{\abs{x-y}\leq \norm{m} + \norm{n} \}.\\
	\end{equation*}
\end{example}

Although there is a natural generalization of the measures $\poislambdaz$ to $\pi^z$, where $\pi^z$ is the distribution of a marked Poisson point process with intensity measure $z\lambda(\drm x) \bigotimes \markdist(\drm m)$, we cannot generalize the definition of a finite-volume Gibbs measure to an infinite-volume Gibbs measure.

For energy function $H$ and $\Lambda \in \boreldom$ define \textit{the conditional energy of $\gamma \in \Mset$ in $\Lambda$} given its environment as
	\begin{equation}\label{def: conditionalenergy}
		H_\Lambda(\gamma) = \lim_{n\to\infty} H(\ms{\gamma}{\Lambda_n}) - H(\ms{\gamma}{\Lambda_n \setminus \Lambda}),
	\end{equation}
where $\Lambda_n = [-n,n)^d$. For the conditional energy to be well defined, we later pose some assumptions on~$H$ (see Section \ref{sec:GMIntro}). For $\Lambda \in \boreldom$, $z > 0$,  energy function $H$ and $\xi \in \Mset$, define \textit{the Gibbs probability kernel} as
	\begin{equation}\label{df:GibbsianKernel}
		\Xi_\Lambda(\xi,\drm \gamma) = \frac{e^{-H_\Lambda(\gamma_\Lambda\xi_{\Lambda^c})}}{Z_\Lambda(\xi)}\, \poislambdaz(\drm \gamma),
	\end{equation}
	where $Z_\Lambda(\xi) = \int e^{-H_\Lambda(\gamma_\Lambda\xi_{\Lambda^c})} \poislambdaz(\drm \gamma)$ is the normalizing constant. Again, for $\Xi_\Lambda(\xi,\drm \gamma)$ to be well defined, we need $0<Z_\Lambda(\xi)<\infty$. Under our assumptions on $H$ (see~Section \ref{sec:GMIntro}~), this will be true for $\xi \in \Msettemp$. The  infinite-volume Gibbs measure is now defined as a~probability measure on $\Mset$, which satisfies the $\DLR$ equations (named after Dobrushin, Lanford and Ruelle). This definition follows naturally from the fact that the finite-volume Gibbs measures also satisfy $\DLR$ (see \cite{ar:DereudreIntro}, Proposition 5.3).

\begin{definition}\label{def:infinitegibbsmeasure}
	A probability measure $\mathsf{P}$ on $\Mset$ is called \textit{an infinite-volume Gibbs measure} with energy function $H$ and activity $z$, if for all $ \Lambda \in \boreldom$ and for all measurable bounded local functions $F: \Mset \to \R$ the $\DLR_\Lambda$ equation holds:
	\begin{equation*}
		\int_\Mset F(\gamma) \, \mathsf{P}(\drm \gamma) = \int_\Mset \int_\MsetLambda F(\gamma_\Lambda\xi_{\Lambda^c})\,\Xi_\Lambda(\xi, \drm \gamma)\, \mathsf{P}(\drm \xi).
	\end{equation*}
\end{definition}

The existence of a measure satisfying the definition above is not guaranteed and must be proved.

\section{The existence result}\label{sec:GMIntro}
To be able to prove the existence of an infinite-volume Gibbs measure, we need the following four assumptions: the moment assumption $\hmoment$, the stability assumption $\hs$, the local stability assumption $\hl$ and the range assumption $\hr$.

\subsection{The Moment and Stability Assumptions} \label{subsec:StabilityAs}
Recall that we have chosen $\delta > 0$ and a reference mark distribution $\markdist$ in Section \ref{subsec:TCaGM}. \textit{The moment assumption} $\hmoment$ concerns the distribution $\markdist$, as it has to satisfy
\begin{equation*}
	\hmoment: \int_{\markspace} \exp(\norm{m}^{d+2\delta}) \markdist(\drm m) < \infty.
\end{equation*}

The rest of the assumptions pertain to the energy function $H$. The first one is  \textit{the stability assumption}:
\begin{equation*}
	\mathcal{H}_{s}: \text{There exists } c\geq0 \text{ such that } \forall \gamma \in \Msetf: H(\gamma) \geq -c \left< \gamma,1 + \norm{m}^{d+\delta} \right>. \\
\end{equation*}

Under the assumptions $\hs$ and $\hmoment$ we have $0<Z_\Lambda<\infty$ for all $\Lambda \in \boreldom$ and therefore the finite-volume Gibbs measures are well defined (for proof, see Lemma 4 in \cite{ar:RoellyZass}). For the infinite-volume Gibbs measure to be well defined, we need an analogue of the stability assumption for the conditional energy, \textit{the local stability assumption}.
\begin{align*}
	\mathcal{H}_{l}: &\text{ For all } \Lambda \in \boreldom \text{ and all }t \in \N \text{ there exists } c(\Lambda,t) \geq 0 \text{ such that for all } \xi \in \Msett \\ &\text{ the following inequality holds for any } \gamma_\Lambda \in \MsetLambda: \\
	& \qquad \qquad \qquad \quad H_\Lambda(\gamma_\Lambda\,\xi_{\Lambda^c}) \geq -c(\Lambda,t) \left<\gamma_\Lambda,1 +
	\norm{m}^{d+\delta}\right>.
\end{align*}

Let us emphasize that the lower bound for the conditional energy must hold uniformly over $\Msett$. In the same way the stability ensures that the partition function is finite, it can be proved that assumptions $\hl$ and $\hmoment$ imply that $0 < Z_{\Lambda} (\xi) < \infty$, for all $\Lambda \in \boreldom$ and $\xi \in \Msettemp$ (for proof, see Lemma 7 in \cite{ar:RoellyZass}).
Contrary to the stability assumption, local stability is often, but not automatically, satisfied for non-negative energy functions. We state our observation regarding the validation of this assumption.

\begin{observation} \label{claim:LokStabNerapornost}
	Assume that the energy function $H$ satisfies $H(\gamma_A) - H(\gamma_B) \geq 0$, for all $\gamma \in \Msetf$ , whenever $B \subset A$; $A,B \in \boreldom$. Then the conditional energy is non-negative and the local stability assumption $\hl$ holds.
\end{observation}	

\begin{proof}
	Let $\gamma \in\Mset$ and $\Lambda \in \boreldom$. Then $\exists K \in \N$ and points $\boldsymbol{x_1}, \dots, \boldsymbol{x_K} \in \rd \times \markspace$ such that $\gamma_{\Lambda} = \sum_{i=1}^{K} \delta_{\boldsymbol{x_i}}$ and we can write
	\begin{align*}
		\lim_{n\to\infty} H(\gamma_{\Lambda_n}) - H(\gamma_{\Lambda_n \setminus \Lambda})
		= \lim_{n\to\infty} \sum_{i=1}^{K} H(\gamma_{\Lambda_n \setminus \{x_1, \dots, x_{i-1}\}}) - H(\gamma_{\Lambda_n \setminus \{x_1, \dots, x_{i}\}}) \geq 0.
	\end{align*}
\end{proof}

\subsection{New formulation of the range assumption}
The last assumption
considers the range of the  interactions among the points.
First, we state the original range assumption from \cite{ar:RoellyZass}.
\begin{align*}
	\tilde{\hr}: &\text{ Fix }\Lambda \in \boreldom. \text{ For any }\gamma \in \Msett, \, t \geq 1, \, \text{ there exists }  \tau(\gamma,\Lambda) > 0 \text{ such that } \\
	& \qquad \qquad H_\Lambda(\gamma) = H(\gamma_{\Lambda\oplus B(0,\tau(\gamma,\Lambda))}) - H(\gamma_{(\Lambda\oplus B(0,\tau(\gamma,\Lambda)))\setminus\Lambda}),
\end{align*}
 where $\Lambda\oplus B(0,R) = \{x \in \rd: \exists y \in \Lambda,\, \abs{x-y} \leq R\}$.
It is noted that the choice of $\tau(\gamma, \Lambda)$ can be
\begin{equation} \label{def:wrongtau}
	\tau(\gamma, \Lambda) = 2\l(t) + 2\supmark{\gamma_{\Lambda}} + 1,
\end{equation}
and this choice is used in the proof of the existence theorem.
	Contrary to the claims in \cite{ar:RoellyZass}, this choice of $\tau(\gamma,\Lambda)$ does not work for the presented examples of the energy function, as we prove in the following lemma.

\begin{lemma} \label{lemma:protiprikladproHr}
	Consider the state space $\rtwo \times \R$ and the energy function $H_1$ from Example~\ref{Ex:EnergyFuncPair}.
	Then $\forall \delta > 0$ there exist $\Lambda \in \boreldvaom$ and a set $\Mset_C \subset \Mset^1$ such that $\forall \gamma \in \Mset_C$
	\begin{equation*}
		\lim_{n\to\infty} H(\gamma_{\Lambda_n}) - H(\gamma_{\Lambda_n \setminus \Lambda}) = H_\Lambda(\gamma) \neq H(\gamma_{\Lambda\oplus B(0,\tau)}) - H(\gamma_{(\Lambda\oplus B(0,\tau))\setminus \Lambda})
	\end{equation*}
	if we choose $\tau = 2l(1) + 2\mathsf{m}(\gamma_\Lambda) +1$.
\end{lemma}
\begin{proof}
	First, we consider $\delta = \frac{1}{2}$  and afterwards modify the example for general $\delta$.
	\\
	\emph{Step 1)} Let $\delta = \frac{1}{2}$. It holds that (see Lemma 2 in \cite{ar:RoellyZass})
	\begin{equation*}
		l(t) =  
		\frac{1}{2} \cdot t^{\frac{1}{\delta}} \cdot 2^{\frac{2+\delta}{\delta}} = t^2\cdot 2^{4}.
	\end{equation*}
	Therefore, we get that $l(1) = 2^4 = 16$. Take points $(x,m),(y,n) \in \rtwo \times \R$, where $x = (120,120)$, $m = 1$, $y = (150,150)$ and $n = 43$. Let $\Lambda = B(x,\varepsilon)$, where $\varepsilon \in \left[0,1\right]$, and set $\gamma = \delta_{(x,m)}+ \delta_{(y,n)}$. Then it holds that $\gamma \in \Mset^1,$ $B(x,m) \cap B(y,n) \neq \emptyset,$ and $(y,n) \notin \gamma_{\Lambda\oplus B(0,\tau)}$ for $\tau = 2l(1) + 2\mathsf{m}(\gamma_\Lambda) +1 = 35$,
and therefore
	\begin{align*}
		&\lim_{n\to\infty} H(\gamma_{\Lambda_n}) - H(\gamma_{\Lambda_n \setminus \Lambda}) = \phi(x,y),\\
		&H(\gamma_{\Lambda\oplus B(0,\tau)}) - H(\gamma_{(\Lambda\oplus B(0,\tau))\setminus \Lambda}) = 0-0 = 0.
	\end{align*}
	Choose $k \in \N$ such that $k \geq \l(1)$ and $\Lambda \oplus B(0,1) \subset U(0,k)$ and then define the set $\Mset_C = \{ \gamma\,\xi_{U(0,2k+1)^c}: \xi \in \Mset^1, (x,m) \in \xi, (y,n) \in \xi\} \subset \Mset^1$. We get that $\forall \nu \in \Mset_C$
	\begin{align*}
		&\lim_{n\to\infty} H(\nu_{\Lambda_n}) - H(\nu_{\Lambda_n \setminus \Lambda}) = \phi(x,y),\\
		&H(\nu_{\Lambda\oplus B(0,\tau)}) - H(\nu_{(\Lambda\oplus B(0,\tau))\setminus \Lambda}) = 0-0 = 0.
	\end{align*}
	\emph{Step 2)} Let $\delta > 0$. Then we can choose $(x,m)$ and $(y,n)$ in the following way:
	\begin{enumerate}
		\item $m = 1$ and $x = (x',0)$  where $x'$ is large enough so that  $$4+2\cdot2^{\frac{2}{\delta}} \leq \frac{1}{2} + (( 4+ 2\cdot2^{\frac{2}{\delta}} + x')^2 - 3)^{\frac{1}{2+\delta}}
		\text{  and  } x' > 1,$$
		\item $y = (y',0)$, where $y' = 4+ 2\cdot2^{\frac{2}{\delta}} + x'$ and $n = (( y')^2 - 3)^{\frac{1}{2+\delta}}.$
	\end{enumerate}
	
	Set $\gamma = \delta_{(x,m)}+ \delta_{(y,n)}$ and $\Lambda = B(x,\varepsilon)$ for some $\varepsilon \in \left[0,\frac{1}{2}\right]$. We again obtain that $\gamma~\in~\Mset^1$, $B(x,m) \cap B(y,n) \neq \emptyset$ and also that $(y,n) \notin \gamma_{\Lambda\oplus B(0,\tau)}$ for the choice $\tau = 2l(1) + 2\mathsf{m}(\gamma_\Lambda) +1$.
	The choice of $\Mset_C$ proceeds in the same way as in the first step.	\\ \makebox{}
\end{proof}

\medskip

In particular, we have found a counterexample to the claim that for any configuration $\gamma_\Lambda\in\MsetLambda$ and any $\xi \in \Msett$ the equality  $H_\Lambda(\gamma_\Lambda \xi_{\Lambda^c}) = H_\Lambda(\gamma_\Lambda \xi_{\Delta \setminus \Lambda})$ holds as soon as $\Lambda~\oplus~B(0,2l(t) + 2\supmark{\gamma_{\Lambda}} +1) \subset \Delta$. We have given the counterexample for $d=2$, $t=1$ and pairwise-interaction model, however, it should be clear that it would be possible to find counterexamples in the same way for $d \geq 3$, $t \in \N$ and other energy functions, for which the interaction between two points is given by the intersection of their respective balls\footnote{After seeing our counterexample from Lemma \ref{lemma:protiprikladproHr}, the authors of \cite{ar:RoellyZass} submitted errata with a corrected form of formula (\ref{def:wrongtau}) by means of adding a term that depends on the distance from
	the origin.}.

We propose the following modification of the range assumption.

\medspace

$\hr$: Fix $\Lambda \in \boreldom$ and $\l \in \N.$ Then for all $\gamma \in \Msettemp$ such that $\gamma_{\Lambda^c} \in \Msetl{\l}$ there exists $\tau = \tau(\supmark{\gamma_\Lambda},l,\Lambda) > 0$ such that
\begin{equation*} \label{formula:rangeassumption}
	H_\Lambda(\gamma) = H(\gamma_{\Lambda\oplus B(0,\tau)}) - H(\gamma_{(\Lambda\oplus B(0,\tau))\setminus\Lambda}),
\end{equation*}
holds and $\tau(\supmark{\gamma_\Lambda},l,\Lambda)$ is a non-decreasing function of $\supmark{\gamma_{\Lambda}}$. Particularly, $\tau$ depends on $\gamma_{\Lambda^c}$ only through $\l$.

With this new modified assumption the existence theorem from \cite{ar:RoellyZass} holds.

\begin{theorem}[Theorem 1 in \cite{ar:RoellyZass}] \label{th:Existencetheorem}
	Under assumptions $\mathcal{H}_{s}$, $\mathcal{H}_{l}$, $\mathcal{H}_r$ and $\mathcal{H}_m$ there exists at least one infinite-volume Gibbs measure with energy function $H$.
\end{theorem}

Concerning the use of the range assumption in the original proof in \cite{ar:RoellyZass}, we refer to the proof of formula (20) on page 990 and the estimation of $\delta_{24}$ on page 992, which remain the same.

\section{Gibbs facet process}\label{subsec:Facets}
The first model we consider will be the process of facets (presented in \cite{ar:VeceraBenesFasety}). For $d \geq 2$  denote by $\gd$ the space of all $(d-1)$-dimensional linear subspaces of $\rd$ and let $\unitsphere$ denote the unit sphere in $\rd$, then $A(n) \in \mathcal{G}_d$ denotes the linear subspace with unit normal vector $n \in \unitsphere$. Let $A(n) \in  \gd$ and $R > 0$. Then a \textit{ facet} $V(n,R)$ with  \textit{radius} $R$ and \textit{normal vector} $n$ is defined as $V(n,R) =  A(n) \cap B(0,R)$.

  As we can see from the definition above, each facet is uniquely described by its radius $R$ and its normal vector $n$ (up to the orientation of $n$). Therefore, it is natural to choose the space of marks as $(\markspace, \norm{\cdot}) = (\mathbb{R}^{d+1}, \norm{\cdot})$ with the standard Euclidean norm\footnote{We will use the notation $\norm{m}$ for the Euclidean norm when talking about mark $m$ from $\rd$ and $\abs{x}$ when talking about location point $x$ from $\rd$.} $\norm{m} = \sqrt{\sum_{i=1}^{d+1} m_i^2}$. The marks will be specified by choosing the reference mark distribution $\markdist$ so that it satisfies
$	\markdist(\hem\times(0,\infty)) = 1,$
where $\hem$ is the \textit{semi-closed unit hemisphere} in $\rd$,
$$\hem = \bigcup_{i=1}^{d}\{u \in \unitsphere: u_1 = 0, \cdots, u_{i-1}=0, u_i > 0\}.$$

Clearly, there exists a bijection between $\hem\times(0,\infty)$ and the set of all facets, the first $d$ coordinates define the normal vector for the corresponding facet and the last coordinate defines the radius.
Take $\gamma \in \Mset$, then we define the \textit{set of all facets} (shifted to their location) \textit{corresponding to configuration $\gamma$} as
	$
		\mathcal{A}(\gamma) = \{x +  V(n,R) : (x,n,R) \in \gamma \}
	$.
The energy of a configuration will depend on the number (and volume) of intersections among the facets. \textit{The energy function of a facet process} is defined as
	\begin{equation} \label{def:FceEnergieFacety}
		\begin{aligned}
			H(\gamma) &= \sum_{j=2}^{d} a_j \phi_j(\gamma), \text{ 	where $a_2, \dots, a_d \in \R$ and } \\
			\phi_j(\gamma) &= \sumne_{K_1, \dots, K_j \in \mathcal{A}(\gamma)} \mathbb{H}^{d-j} \left(\bigcap_{i=1}^j K_i\right) \cdot \ind \left[ \mathbb{H}^{d-j} (\bigcap_{i=1}^j K_i) < \infty \right].
		\end{aligned}
	\end{equation}
	Here, $\mathbb{H}^{k}$ denotes the $k$-dimensional Hausdorff measure on $\rd$ and $\sum_{...}^{\neq}$ denotes the sum over all $j$-tuples from $\mathcal{A}(\gamma)$. From now on, the indicator in (\ref{def:FceEnergieFacety}) will be denoted by $\ind_{\infty,j}$.

\subsection{Existence of Gibbs facet process with repulsive interactions}

To verify the existence of the Gibbs facet process, we must verify the assumptions of Theorem \ref{th:Existencetheorem}. Regarding the assumption $\mathcal{H}_m$, we have to choose the mark distribution $\markdist$ such that
\begin{align*}
	  \int_{\mathbb{S}^{d-1}_+\times (0,\infty)} \exp((1+R^2)^{d/2+\delta}) \markdist (\drm (n,R)) < \infty.
\end{align*}

To address the assumption $\hr$, we rewrite the definition of sets $\Msetl{\l}$  in the language of facets.
\begin{lemma} \label{lemma:FacetyVlastMl}
	Take $\gamma \in \Msettemp$, $\gamma \in \Msetl{\l_0}$. Then  $\forall \l \geq l_0$ and for all $K \in \mathcal{A}(\gamma)$ we have the following implication: $K \in \mathcal{A}(\ms{\gamma}{U(0,2\l+1)^c}) \implies K \cap U(0,\l) = \emptyset$.
\end{lemma}

\begin{proof}
	We know that the implication
	$(x,m) \in \ms{\gamma}{{U(0,2l+1)}^c} \implies U(0,l) \cap B(x,\norm{m}) = \emptyset$ holds $\forall l \geq \l_0$ (from the definition of $\Msetl{\l_0}$ ) and in our case $ B(x,\norm{m})= B(x,\sqrt{1+R^2})$. Clearly $K \subset B(x,\sqrt{1+R^2})$ and therefore $K \cap U(0,\l) = \emptyset$.
\end{proof}

\bigskip

Now we will show that the range assumption holds.

\begin{theorem} \label{th:FacetyRange}
	The energy function $H$ of a facet process defined in (\ref{def:FceEnergieFacety}) satisfies the range assumption $\hr$.
\end{theorem}
\begin{proof}
	Fix $\Lambda \in \boreldom$ and $\l_0\in \N$. We want to prove that for all $\gamma \in \Msettemp$ such that $\gamma_{\Lambda^c} \in \Msetl{\l_0}$ there exists $\tau = \tau(\supmark{\gamma_{\Lambda}},\l_0,\Lambda) > 0$ which is a non-decreasing function of $\supmark{\gamma_{\Lambda}}$ and for which it holds that
$H_\Lambda(\gamma) = H(\gamma_{\Lambda\oplus B(0,\tau)}) - H(\gamma_{(\Lambda\oplus B(0,\tau))\setminus\Lambda}).$

	Take $i_0 \in \N$ large enough so that $\Lambda \subset \Lambda_{i_0} = \left[-i_0,i_0\right) ^d$. From the definition of the conditional energy, we have that 	
	$	H_\Lambda(\gamma) = \lim_{n\to\infty} \left( H(\ms{\gamma}{\Lambda_n}) - H(\ms{\gamma}{\Lambda_n \setminus \Lambda}) \right).
	$
	
	We can write for all $n \geq i_0$
	\begin{align*}
		H(\ms{\gamma}{\Lambda_n}) - H(\ms{\gamma}{\Lambda_n \setminus \Lambda}) =  \sum_{j=2}^{d} &a_j \sumne_{K_1, \dots, K_j \in \mathcal{A}(\ms{\gamma}{\Lambda_n})} \mathbb{H}^{d-j} \left( \bigcap_{i=1}^j K_i\right) \cdot \ind_{\infty,j} \\
		& - \sum_{j=2}^{d} a_j \sumne_{K_1, \dots, K_j \in \mathcal{A}(\ms{\gamma}{\Lambda_n \setminus \Lambda})} \mathbb{H}^{d-j} \left( \bigcap_{i=1}^j K_i\right)  \cdot \ind_{\infty,j}.
	\end{align*}
Now define for general sets $A\subset B,\, A, B \in \boreldom$, and for $j = 2, \dots, d$ the set of all $j$-tuples of points from $\gamma$ in $B$ (or more specifically the set of (non-ordered) $j$-tuples of facets represented by these points) such that at least one of these points lies in $A$:
	\begin{equation*}
		\begin{aligned}
			\mathcal{C}_j(\gamma, A,B) = \{ \{K_1, \dots, K_j\} : K_i \in \mathcal{A}(\gamma_B) &\text{ for all } i = 1,\dots,j \text{ and } \exists i \text{ s.t. } K_i \in \mathcal{A}(\gamma_A)\}.
		\end{aligned}
	\end{equation*}
	Then for any $\tau > 0$ and $n$ large enough so that $\Lambda \oplus B(0,\tau) \subset \Lambda_n$ we can write
	
	\begin{equation}\label{vzPomR}
		\begin{aligned}
			&H(\ms{\gamma}{\Lambda_n}) - H(\ms{\gamma}{\Lambda_n \setminus \Lambda}) = \sum_{j=2}^{d} a_j \sum_{\{K_1, \dots, K_j\} \in \mathcal{C}_j(\gamma, \Lambda, \Lambda \oplus B(0,\tau))} \mathbb{H}^{d-j} \left( \bigcap_{i=1}^j K_i\right) \cdot \ind_{\infty,j} \\
			& \qquad \qquad + \sum_{j=2}^{d} a_j \sum_{\{K_1, \dots, K_j\} \in \mathcal{C}_j(\gamma, \Lambda,\Lambda_n)\setminus \mathcal{C}_j(\gamma, \Lambda,\Lambda \oplus B(0,\tau))} \mathbb{H}^{d-j} \left( \bigcap_{i=1}^j K_i\right) \cdot \ind_{\infty,j}.
		\end{aligned}
	\end{equation}
	Clearly, the first sum does not depend on $n$ and it is in fact equal to the desired $H(\gamma_{\Lambda\oplus B(0,\tau)}) - H(\gamma_{(\Lambda\oplus B(0,\tau)) \setminus \Lambda)})$. Therefore, it is sufficient to show that for the right choice of $\tau$ each summand in the second sum is $0$. Consider the following steps.

	\begin{enumerate}
		\item We have $\l_1(\supmark{\gamma_{\Lambda}},\Lambda) = \min \{\l \in \N: \Lambda\oplus B(0,\supmark{\gamma_\Lambda}) \subset U(0,\l) \} < \infty$,  since $\supmark{\gamma_\Lambda}$ is finite.
		\item Let $\l_2(\supmark{\gamma_{\Lambda}},l_0,\Lambda) = \text{max}\{\l_0,\l_1(\supmark{\gamma_{\Lambda}},\Lambda)\}$.
		\item Take
		$\tau(\supmark{\gamma_{\Lambda}},l_0,\Lambda) = \min \{ k \in \N: U(0,2\l_2(\supmark{\gamma_{\Lambda}},l_0,\Lambda)+1) \subset \Lambda \oplus B(0,k) \}.$
	\end{enumerate}
	
	Then clearly for $a < b$ we have that $\tau(a,l_0,\Lambda) \leq \tau(b,l_0,\Lambda)$. Now, let $n_0$ be the smallest $n$ such that $\Lambda \oplus B(0,\tau(\supmark{\gamma_{\Lambda}},l_0,\Lambda)) \subset \Lambda_n$. Let $n \geq n_0$ and fix $j \in \{2, \dots, d\}$. For simplicity, denote $\tau = \tau(\supmark{\gamma_{\Lambda}},l_0,\Lambda)$ and $l_2 = \l_2(\supmark{\gamma_{\Lambda}},l_0,\Lambda)$ from the second step in the definition of $\tau$.
	
	Take $\{K_1, \dots, K_j\} \in \mathcal{C}_j(\gamma, \Lambda,\Lambda_n) \setminus \mathcal{C}_j(\gamma, \Lambda, \Lambda \oplus B(0,\tau))$. From the definition of $\mathcal{C}_j$ there exist indices $i,k \in \{1, \dots, j\}, i \neq k$ such that $K_i  = x +  V(n,R) \in \mathcal{A}(\gamma_\Lambda)$ and $K_k \in \mathcal{A}(\gamma_{\Lambda_n \setminus \Lambda \oplus B(0,\tau) })$.

	In particular, considering the choice of $\tau$ above, it holds that
 $K_i \subset U(0,\l_2)$ (from the first and second step),
 $K_k \in \mathcal{A}(\ms{\gamma}{U(0,2\l_2+1)^c})$ (from the third step) and
 $\l_2 \geq \l_0$ (from the second step). We get from Lemma \ref{lemma:FacetyVlastMl} that $K_i \cap K_k = \emptyset$ and so $\mathbb{H}^{d-j} \left( \bigcap_{i=1}^j K_i\right) = 0$.
	This holds for all $\{K_1, \dots, K_j\} \in \mathcal{C}_j(\gamma, \Lambda,\Lambda_n) \setminus \mathcal{C}_j(\gamma, \Lambda, \Lambda \oplus B(0,\tau))$ and therefore we have for all $ n \geq n_0$ and for $\tau = \tau(\supmark{\gamma_{\Lambda}},l_0,\Lambda)$:
	$$
	H(\ms{\gamma}{\Lambda_n}) - H(\ms{\gamma}{\Lambda_n \setminus \Lambda}) \overset{(\ref{vzPomR})}{=} \sum_{j=2}^{d} a_j \sum_{\{K_1, \dots, K_j\} \in \mathcal{C}_j(\gamma, \Lambda, \Lambda \oplus B(0,\tau))} \mathbb{H}^{d-j} \left( \bigcap_{i=1}^j K_i\right) \cdot \ind_{\infty,j}.
	$$
\end{proof}

\bigskip

Consider now a situation where the constants $a_j$ in the definition of the energy function of facet process satisfy $a_j \geq 0$ for all $j =2, \dots, d$. This leads to repulsive interactions between the facets, i.\,e., configurations with a lot of interacting facets will have higher energy, compared to those with disjoint facets, and they will be therefore less probable. We have the following existence result.

\begin{theorem}
	Let the energy function (\ref{def:FceEnergieFacety}) of a facet process satisfy $a_j\geq 0,$ for all $ j \in \{2, \dots, d\}$ and assume that the reference mark distribution $\markdist$ satisfies $\hmoment$. Then the infinite-volume Gibbs facet process exists.
\end{theorem}
\begin{proof}
	In this case, the energy function $H$ is non-negative and, therefore, the stability assumption $\hs$ holds. Since $H$ clearly satisfies observation \ref{claim:LokStabNerapornost}, the local stability assumption $\hl$ also holds. Theorem \ref{th:FacetyRange} shows that also the range assumption $\hr$ is satisfied and therefore the assumptions of Theorem \ref{th:Existencetheorem} hold and the existence is proven.	\\
	\makebox{}
\end{proof}

\subsection{The counterexample for attractive interactions in $\rtwo$}

Consider the facet process in $\mathbb{R}^2$, i.\,e., the energy function is
\begin{equation} \label{vzFceEnFac2}
	H(\gamma) = a_2 \sumne_{K_1, K_2 \in \mathcal{A}(\gamma)} \mathbb{H}^{0} \left( K_1 \cap K_2 \right)  \cdot \ind_{\infty,2}.
\end{equation}
Suppose that $a_2 < 0$ (we can assume for simplicity that $a_2 = -1$). This leads to attractive interactions between the facets. We will show that the finite-volume Gibbs measures do not exist. The first step will be to find a sequence $\{\gamma_N\}_{N\in \N} \subset \Msetf$ contradicting the stability assumption $\hs$. In the second step, we show that we can modify these configurations (under some mild assumptions on the mark distribution $\markdist$) to form a~sequence of subsets $A_{\Lambda,N} \subset \Msetf$ for some $\Lambda \in \boreldvaom$, such that $\poislambdaz(A_{\Lambda,N}) > 0$, $\forall N \in \N$, and $\hs$ does not hold on $\bigcup_{N \in \N} A_{\Lambda,N}$. In the final step, we use sets $A_{\Lambda,N}$ to show that the partition function $Z_\Lambda$ is infinite.

\medspace

\noindent\emph{Step 1)} 

Consider the following lemma.
\begin{lemma}
	The energy function of a facet process in $\rtwo$ (i.\,e., (\ref{vzFceEnFac2})) does not satisfy the stability assumption $\hs$ for $a_2 < 0$.
\end{lemma}
\begin{proof}
	Take $N \in \N$ even, $n_1, n_2 \in \hemdva$ and $R > 0$ and $\gamma_N \in \Msetf$ satisfying
	\begin{itemize}
		\item[i)] $\textbf{supp }\gamma_N = \{(x_1, n_1, R), \dots, (x_{\frac{N}{2}}, n_1, R), (x_{\frac{N}{2}+1}, n_2, R), \dots ,(x_{N}, n_2, R) \},$
		\item[ii)] normal vectors $n_1, n_2 \in \hemdva$ satisfy $n_1 \neq n_2$,
		\item[iii)] the location points satisfy $x_i = (x_i^1,0)^T$, where $1 = x_1^1 > x_2^1 > \dots > x_{\frac{N}{2}}^1 > 0$ and $-1 = x_{\frac{N}{2}+1}^1 < x_{\frac{N}{2}+2}^1 < \dots < x_{N}^1 < 0$,
		\item[iv)] $R > 0$ is a sufficiently large constant (depending on $n_1,n_2$) such that the facets  $x_1~+~V(n_1,R)$ and $x_{\frac{N}{2}+1} +  V(n_2,R)$ intersect.
	\end{itemize}

	 It holds for these configurations that each facet given by the points $(x_i, n_1, R)$, $i~\in~\{1, \dots, \frac{N}{2}\}$, intersects all facets given by the second half of the points and there are no intersections within the first half and within the second half. So, we have that
	\begin{equation*}
		H(\gamma_N) = - \sumne_{K_1, K_2 \in \mathcal{A}(\gamma_N)} \mathbb{H}^{0} \left( K_1 \cap K_2 \right)  \cdot \ind_{\infty,2} = - \frac{N}{2} \cdot \frac{N}{2}.
	\end{equation*}
	At the same time
$	\left< \gamma_N,1 + \norm{m}^{2+\delta} \right> = \sum_{i = 1}^{N} (1+ (1+R^2)^{1+ \frac{\delta}{2}}) = N \cdot (1+ (1+R^2)^{1+ \frac{\delta}{2}}).
$
	Denote by $b := (1+ (1+R^2)^{1+ \frac{\delta}{2}}) < \infty$ the constant, which does not depend on $N$.
	Assume for contradiction that $\hs$ holds, i.\,e., there exists $ c > 0$ such that $\forall \gamma \in \Msetf$ we have $H(\gamma) \geq -c \left< \gamma,1 + \norm{m}^{2+\delta} \right>$. Then we get that $\forall N \in \N$ even
$
		- \frac{N}{2} \cdot \frac{N}{2} \geq - c\cdot N \cdot b,
$
	which is clearly a contradiction.
\end{proof}

\bigskip

\noindent\emph{Step 2)} 

From now on, we assume that there exist vectors $u,v \in \hemdva$, some constants $0 < a \leq b < \infty$ and $\eps > 0$ such that
\begin{equation}
	\begin{aligned} \label{podmnaRprotipriklad}
		\markdist(U(u \pm \eps)\times(a,b)) > 0, \, \markdist(U(v \pm \eps)\times(a,b)) > 0 \text{ and }
		U(u \pm \eps) \cap U(v \pm \eps) = \emptyset,
	\end{aligned}
\end{equation}
where $U(u \pm \eps) = \{w \in \hemdva : \abs{\sphericalangle(u,w)} \leq \eps \}$ (here $\sphericalangle (u,w)$ denotes the angle between the vectors $u$ and $w$).  We are able to find a set $\Lambda$ such that if two facets have centres inside $\Lambda$, their normal vectors do not differ too much from $u$
and $v$, respectively, and their length is at least $a$, then they must intersect at one point.

\begin{lemma} \label{pomlemma:volbaLambda}
	For given constants $a, \eps > 0$ and two different vectors $u,v \in \hemdva$ there exists a set $\Lambda \in \boreldvaom$ such that
	$$\abs{(x +  V(n,R)) \cap (y +  V(m,T))} = 1$$
	holds for all $ x,y \in \Lambda, x \neq y$, $n \in U(u \pm \eps)$, $m \in U(v \pm \eps)$ and $R>a,T>a$.
\end{lemma}

\begin{proof}
	Set $\Lambda_0 = \left[-1,1\right]^2$ and take $x, y \in \Lambda_0$, $ n \in U(u \pm \eps)$, $n = (n_1,n_2)^T$, and $m \in U(v \pm \eps)$, $m = (m_1,m_2)^T$. We will denote by $\dotproduct{x}{y}$ the standard dot product on~$\rtwo$. Denote by
	$
	p(x,n) = \{z \in \rtwo : \dotproduct{z}{n} = \dotproduct{n}{x}\}
	$ the line given by a point $x$ and a normal vector $n$ and analogously the line $p(y,m)$ given by a point $y$ and a normal vector $m$. Then, due to the assumption (\ref{podmnaRprotipriklad}), $n \neq \pm m$ and these two lines intersect at one point
	$P(x,y,n,m) = A^{-1}b,$ where $b = (\dotproduct{n}{x},\dotproduct{m}{y})^T$ and
	$$A = \begin{pmatrix}
		n_1 & n_2  \\
		m_1 & m_2
	\end{pmatrix}.$$
	Then we can define a function $f_1$ as the distance from point $x$ to the intersection $P(x,y,n,m)$,
	$$
	f_1(x,y,n,m) = \norm{x - P(x,y,n,m)}.
	$$
	This is a continuous function on $\Lambda_0 \times \Lambda_0 \times U(u\pm \eps ) \times U(v\pm \eps )$, which is a compact subset of $\mathbb{R}^8$. Therefore, the function $f_1$ has a maximum $M_1$ on this set. Analogously, we can define $f_2$ as the distance from point $y$ to the intersection $P(x,y,n,m)$ and there exists its maximum $M_2$ on $\Lambda_0 \times \Lambda_0 \times U(u\pm \eps ) \times U(v\pm \eps )$.
	Now, we only need the following observation. Take any $s > 0$, then
	\begin{align*}
		f_1(sx,sy,n,m) &= \norm{sx - P(sx,sy,n,m)} =
		\norm{sx - A^{-1}(\dotproduct{n}{sx},\dotproduct{m}{sy})^T} \\ &= s \norm{x - A^{-1}b}
		= s f_1(x,y,n,m).
	\end{align*}
	Therefore, the maximum of $f_1$ on $s\Lambda_0 \times s\Lambda_0 \times U(u\pm \eps ) \times U(v\pm \eps )$ is $sM_1$ and analogously the maximum of $f_2$ on $s\Lambda_0 \times s\Lambda_0 \times U(u\pm \eps ) \times U(v\pm \eps)$ is $sM_2$. Now it is enough to find $s > 0$ small enough such that $\max\{sM_1, sM_2\} < a$ and take $\Lambda = s\Lambda_0 = \left[-s,s\right]^2$.
\end{proof}

\bigskip

Now we take $\Lambda$ from Lemma \ref{pomlemma:volbaLambda} and denote
\begin{align*}
	G_u = \Lambda \times U(u \pm \eps)\times(a,b) &\text{ and } \Gamma_u = \left(z\lambda_\Lambda\otimes\markdist\right)(G_u),\\
	G_v = \Lambda \times U(v \pm \eps)\times(a,b) &\text{ and } \Gamma_v = \left(z\lambda_\Lambda\otimes\markdist\right)(G_v),\\
	D = \Lambda \times \markspace \setminus (G_u \cup G_v) &\text{ and } \Delta = \left(z\lambda_\Lambda\otimes\markdist\right)(D).
\end{align*}
Then we define, $\forall k \in \N$, the following set of configurations
\begin{equation} \label{vz:MnozinyALambdaN}
	A_{\Lambda, 2k} = \{\gamma \in \Msetf: \abs{\gamma} = 2k,\, \gamma (G_u) = k,\, \gamma (G_v) = k \} \subset \MsetLambda.
\end{equation}
Due to the assumption (\ref{podmnaRprotipriklad}), it holds that
\begin{align}\label{plstAN}
	\poislambdaz (A_{\Lambda,2k}) = e^{-\Delta} \cdot e^{-\Gamma_u} \cdot \frac{\Gamma_u^{k}}{k!} \cdot e^{-\Gamma_v} \cdot \frac{\Gamma_v^{k}}{k!} > 0
\end{align}
and thanks to Lemma \ref{pomlemma:volbaLambda} we have that $\forall k \in \N$  and $\forall \gamma \in A_{\Lambda, 2k}$
\begin{equation}\label{vz8}	
		H(\gamma) = - \sumne_{K_1, K_2 \in \mathcal{A}(\gamma)} \mathbb{H}^{0} \left( K_1 \cap K_2 \right)  \cdot \ind_{\infty,2} = - k \cdot k.
\end{equation}

\medspace

\noindent\emph{Step 3)} 

So far, we have only shown that the assumption $\hs$ (and consequently $\hl$) is not satisfied for negative $a_2$ and therefore we cannot use Theorem \ref{th:Existencetheorem}. Now we prove that the finite-volume Gibbs measures do not exist.

\begin{theorem}
	Let $a_2 <0$ and assume that the mark distribution $\markdist$ satisfies (\ref{podmnaRprotipriklad}). Then it holds that $Z_{\tilde{\Lambda}} = + \infty$, $\forall \tilde{\Lambda} \in \boreldvaom$, and therefore the finite-volume Gibbs measures do not exist.
\end{theorem}

\begin{proof}
	Take $\Lambda$ from Lemma \ref{pomlemma:volbaLambda} and $A_{\Lambda,2k}$, $k \in \N$, defined in (\ref{vz:MnozinyALambdaN}). Then $\forall k \in \N$
	\begin{align*}
		Z_\Lambda  &= \int_\MsetLambda e^{-H(\gamma)} \poislambdaz(\drm \gamma) \geq  \int_{A_{\Lambda,2k}} e^{-H(\gamma)} \poislambdaz(\drm \gamma) =  \\
		&= e^{k^2} \cdot \poislambdaz (A_{\Lambda,2k}) = e^{k^2} \cdot e^{-\Delta} \cdot e^{-\Gamma_u} \cdot \frac{(\Gamma_u)^{k}}{k!} \cdot e^{-\Gamma_v} \cdot \frac{(\Gamma_v)^{k}}{k!}.
	\end{align*}
	We have used (\ref{plstAN}) and (\ref{vz8}). Thanks to Stirling's formula the right side converges to~$\infty$ with $k \to \infty$ and therefore $Z_\Lambda = \infty$.
	Now take any $\tilde{\Lambda} \in \boreldom$. Since $H$ is assumed to be translation invariant, we can, without loss of generality, assume that there exists a~constant $1 \geq t > 0$ such that $t\Lambda \subset \tilde{\Lambda}$. Returning to the proof of Lemma \ref{pomlemma:volbaLambda}, we could have used the approach from Step 2) for $t\Lambda$ and everything would have worked in the same way, so we can assume, without loss of generality, that $\Lambda \subset \tilde{\Lambda}$.

	Now denote $\tilde{D} = \tilde{\Lambda} \times \markspace \setminus (G_u \cup G_v)$ and $\tilde{\Delta} = (z\lambda_{\tilde{\Lambda}}\otimes\markdist)(\tilde{D})$. Then we can write
	\begin{align*}
		Z_{\tilde{\Lambda}}  &= \int_{\Mset_{\tilde{\Lambda}}} e^{-H(\gamma)} \pi_{\tilde{\Lambda}}^z(\drm \gamma) \geq  \int_{A_\Lambda,2k} e^{-H(\gamma)} \pi_{\tilde{\Lambda}}^z(\drm \gamma) =  \\
		&= e^{k^2} \cdot \pi_{\tilde{\Lambda}}^z (A_{\Lambda,2k}) = e^{k^2} \cdot e^{-\tilde{\Delta}} \cdot  e^{-\Gamma_u} \cdot \frac{(\Gamma_u)^{k}}{k!} \cdot e^{-\Gamma_v} \cdot \frac{(\Gamma_v)^{k}}{k!},
	\end{align*}
	and we can again use Stirling's formula to get that $Z_{\tilde{\Lambda}} = \infty$.
\end{proof}

\section{Gibbs--Laguerre tessellations}\label{subsec:Tessellations}
In this section, we consider Gibbs--Laguerre processes, which present a model for random tessellations of $\rtwo$. Since it is not possible to use the existence theorem from \cite{ar:RoellyZass} to prove that an infinite-volume Gibbs--Laguerre process exists, we considered a particular energy function, and, using some parts of the proof from \cite{ar:RoellyZass}, we were able to derive a~new existence theorem under the assumption that we almost surely see a point.

\subsection{Tessellations and Laguerre geometry}

We say that a set $T = \{C_i: i \in \N\}$, where $C_i \subset \rtwo$, is a \textit{tessellation of $\rtwo$}, if
	\begin{itemize}
		\item[i)] $\mathrm{int}(C_i) \cap \mathrm{int}(C_j) = \emptyset$ for $i \neq j$,
		\item[ii)] $\bigcup_i C_i = \rtwo$ (it is space filling),
		\item[iii)] $ \abs{\{C_i \in T: C_i \cap B \neq \emptyset\}} < \infty $ for all $B \subset \rtwo$ bounded ($T$ is locally finite),
		\item[iv)] the sets $C_i$ (called \textit{cells}) are convex compact sets with interior points.
	\end{itemize}

The cells of a tessellation are convex polytopes (see Lemma 10.1.1 in \cite{bo:StochasticAndIntergralGeometry}).
We define an \textit{edge of cell $C$} as a 1-dimensional  intersection of $C$ with its supporting hyperplanes, and we define a \textit{vertex of cell $C$} as a 0-dimensional intersection of $C$ with its supporting hyperplanes\footnote{See \cite{bo:ConvexBodies}, Section 2.4., for the theoretical background.}. We denote the set of all edges of $C$ by $\Delta_1(C)$ and the set of all vertices of $C$ by $\Delta_0(C)$.
We also define \textit{the set of edges of a tessellation T} as $S_1(T) = \left \lbrace F(y): \mathrm{dim}(F(y)) = 1,\, y \in \rtwo\right \rbrace,$ where  $F(y)$ is the intersection of all cells of~$T$ containing the point $y$, $F(y) = \bigcap_{C \in T:\, y \in C} C.$ Analogously, we could define  $S_0(T)$, the \textit{set of vertices of a tessellation~T}. It always holds that $\bigcup_{C \in T} \Delta_0(C) = S_0(T)$, but it can happen that $\bigcup_{C \in T} \Delta_1(C) \neq S_1(T)$. However, we will not consider such tessellations.

A tessellation $T$ is called \textit{normal} if it satisfies that $\bigcup_{C \in T} \Delta_1(C) = S_1(T)$, every edge is contained in the boundary of exactly two cells and every vertex is contained in the boundary of exactly three cells.

We will now focus on Laguerre diagrams (see \cite{ar:LautensackZuyev} for general theory), which are based on the power distance from some fixed set of weighted points. For $x,z \in \rtwo$ and $u \geq 0$ define \textit{the power distance} of $z$ and weighted point $(x,u)$ as $\rho(z,(x,u)) = \abs{x-z}^2 - u^2$.
Denote for points $x,y \in \rtwo$ and weights $u,v \geq 0$
\begin{align*}
	HP\left((x,u),(y,v)\right) &= \left \lbrace z \in \rtwo: \rho(z,(x,u)) = \rho(z,(y,v))\right \rbrace  \\
	& = \left \lbrace z \in \rtwo: 2\dotproduct{y-x}{z} = \abs{y}^2 - \abs{x}^2+u^2 - v^2\right \rbrace
\end{align*}
the line (called \textit{radical axis}) separating $\rtwo$ into two half-planes
and
\begin{equation}
	\begin{aligned} \label{def:HalfPlaneP}
		P\left((x,u),(y,v)\right)&=\left \lbrace z \in \rtwo: \rho(z,(x,u)) \leq \rho(z,(y,v))\right \rbrace \\
		&= \left \lbrace z \in \rtwo: 2\dotproduct{y-x}{z} \leq \abs{y}^2 - \abs{x}^2+u^2 - v^2\right \rbrace
	\end{aligned}
\end{equation} the closed half-plane, whose points are closer to $(x,u)$ than to $(y,v)$ \wrt\ to the power distance.

Now take at most countable subset $\gamma \subset \rtwo \times (0,\infty)$ of weighted points, which will be called \textit{the set of generators}. We will use the notation $x = (x',x'')$ for $x \in \gamma $, where $x'$ denotes the location and $x''$ the weight of the point.	Assume that $\gamma$ satisfies assumption (R0): $ \forall z \in \rtwo \text{ there exists } \underset{x \in \gamma}{\text{min}} \, \rho(z,x).$

We define \textit{the Laguerre diagram} of $\gamma$ as $$L(\gamma) = \{L(x,\gamma): x \in \gamma,\, L(x,\gamma) \neq \emptyset\},$$ where $L(x,\gamma)$ is \textit{the Laguerre cell of $x$ in $\gamma$} defined as $$L(x,\gamma) = \{z \in \rtwo: \rho(z,x) \leq \rho(z,y) \, \, \forall y \in \gamma\}.$$
	We call $x'$ the \textit{nucleus} of the cell $L(x,\gamma)$ and denote by $\gamma' = \{x': (x',x'') \in \gamma\}$ the set of nuclei of $\gamma$. The set of points from $\gamma$, whose cells are empty, is denoted by $E(\gamma) = \{x \in \gamma : L(x,\gamma) = \emptyset\}$.
Clearly from the definition, the (possibly empty) Laguerre cell can be written as
\begin{equation} \label{def:LagCellAsIntersection}
	L(x,\gamma) = \bigcap_{y \in \gamma} P(x,y).
\end{equation}

The following conditions were derived in \cite{ar:LautensackZuyev} for $L(\gamma)$ to be a tessellation. We say that $\gamma \subset \rtwo \times (0,\infty)$ fulfils \textit{regularity conditions} if it satisfies
	\begin{itemize}
		\item[(R1)] for all $(z,t) \in \rtwo \times \R$ only finitely many  $x \in \gamma$ satisfy $\abs{z-x'}^2 - (x'')^2 \leq t, $
		\item[(R2)]$\text{conv}\{x': (x',x'') \in \gamma\} = \rtwo$.
	\end{itemize}
	
We say that  $\gamma \subset \rtwo \times (0,\infty)$ is \textit{in general position} if the following conditions hold
	\begin{itemize}
		\item[(GP1)] no 3 nuclei are contained in a 1-dimensional affine subspace of $\rtwo$,
		\item[(GP2)] no 4 points have equal power distance to some point in $\rtwo$.
	\end{itemize}

The following can be shown (see \cite{ar:LautensackZuyev}, Theorem 2.2.8.). Let $\gamma$ satisfy (R1) and (R2). Then every cell $L(x,\gamma)$, $x \in \gamma$, is compact, $L(\gamma)$ is locally finite and space filling and $\tilde{L}(\gamma) = \{L(x,\gamma) \in L(\gamma): \mathrm{int}(L(x,\gamma)) \neq \emptyset\}$ is a face-to-face tessellation. If $\gamma$ satisfies (R1),(R2), (GP1) and (GP2), then all cells of $L(\gamma)$ have dimension 2 and the Laguerre diagram $L(\gamma)$ is a normal tessellation.

A finite set of generators will not satisfy condition (R2), but this case can be easily treated separately. Assume that $\gamma \subset \rtwo \times (0,\infty)$ is finite, $\gamma = \{x_1, \dots, x_N\}$ for some $N \in~\N$. Then assumption (R0) holds and therefore the Laguerre cells $L(x,\gamma)$ are well defined.  We see from (\ref{def:LagCellAsIntersection}) that each cell is an intersection of finitely many closed hyperplanes, i.\,e., bounded $L(x,\gamma)$ are convex polytopes. $L(\gamma)$ is space filling and for two points $x_i, x_j \in \gamma$ such that their cells have non-empty interiors, we get that $\mathrm{int}\left(L(x_i, \gamma)\right) \cap \text{int}\left(L(x_j, \gamma)\right) = \emptyset$.
  Let $\Delta_i(x,\gamma)$ denote the set of $i$-dimensional intersections of the cell $L(x,\gamma)$ with the hyperplanes $HP(x,y),\, y \in \gamma$, and define $\Delta_i(\gamma) = \bigcup_{x \in \gamma}\Delta_i(x,\gamma)$, $i=0,1$. The following can easily be proved.

\begin{proposition}
	The diagram $L(\gamma)$ is well defined for a finite set of generators $\gamma$. Assume that $\gamma$ satisfies (GP1) and (GP2). Then it holds that the cell $L(x, \gamma)$, $x \in \gamma$, is either empty or has dimension 2, each vertex $v\in \Delta_0(\gamma)$ lies in the boundary of exactly three cells and each edge $e \in \Delta_1(\gamma)$ lies in the boundary of exactly two cells.
\end{proposition}

For finite $\gamma$ in general position we say that $L(\gamma)$ is a \textit{generalized normal tessellation}.

\subsection{Gibbs--Laguerre measures}\label{Sec:Gibbs-LaguerreMeasures}

To model a random Laguerre diagram $L(\Psi)$, we consider a Laguerre diagram with random set of generators $\Psi$, where $\Psi$ is a marked point process in the space $\rtwo \times (0,\infty)$. Our aim was to consider $\Psi$ to be an infinite-volume marked Gibbs point process with energy function depending on the geometric properties of $L(\Psi)$ and to use Theorem \ref{th:Existencetheorem} to show that there exists an infinite-volume Gibbs--Laguerre measure with unbounded weights. Unfortunately, the range assumption $\hr$ turned out to be an insurmountable obstacle, and our approach needed to be adjusted.

\subsubsection{The energy function and finite-volume Gibbs measures}
Let the state space be $\rtwo \times \mathbb{R}$ with mark space $(\mathbb{R},\norm{\cdot})$ and take a mark distribution $\markdist$  such that $\markdist((0,\infty)) = 1$ and such that $\hmoment$ holds.
We will work with energy function
\begin{equation} \label{def:EFVertices}
	H(\gamma) =  \begin{cases}
		\sum_{x \in \gamma} \abs{\Delta_0(x,\gamma)} & \qquad \text{if } E(\gamma) = \emptyset, \\
		+ \infty & \qquad \text{if } E(\gamma) \neq \emptyset,
	\end{cases}  \qquad \gamma \in \Msetf.
\end{equation}
We sum the number of vertices for each Laguerre cell and we forbid the configurations for which there exists an empty cell.
Clearly, $H$ is non-negative. Therefore, the stability assumption $\hs$ is satisfied and the finite-volume Gibbs measure $\dist{\Lambda}(\drm \gamma)$ (see (\ref{def:FVGM})) in $\Lambda \in \boreldvaom$ with energy function (\ref{def:EFVertices}) and activity $z > 0$ is well defined.

It holds (see Proposition 3.1.5 in \cite{ar:LautensackZuyev}, or \cite{ar:HZGeneralPosition}) that \begin{equation}\label{vz:PoissonGeneralPosition}
	\poislambdaz\left(\{\gamma \in \Mset: \gamma \text{ is in general position}\}\right)=1,\, \forall \Lambda \in \boreldvaom,\, \forall z > 0.
\end{equation}
Therefore, also $\dist{\Lambda}(\{\gamma \in \Mset: \gamma \text{ is in general position}\})=1,$ particularly the Laguerre diagram $L(\gamma)$ is a generalized normal tessellation for $\dist{\Lambda}$-{a.a.}$ \, \gamma$. Since configurations with empty cells are forbidden, we also get that \begin{equation} \label{vz:GibbsNonemptySets}
	\dist{\Lambda}(\{\gamma: E(\gamma) = \emptyset\}) = 1.
\end{equation}

In the following proposition, we present the key observation for the energy function $H$ from (\ref{def:EFVertices}). This observation will later allow us to show that the conditional energy $H_\Lambda$ is attained as soon as all cells belonging to the points in $\Lambda$ are bounded.

\begin{proposition} \label{lemma:KeyPropForEnergyFunctionVertices}
	Let $H$ be the energy function defined in (\ref{def:EFVertices}) and take $\gamma \in \Msetf$ such that it satisfies (GP1), (GP2) and $E(\gamma) = \emptyset$. Assume that the Laguerre cell $L(x, \gamma)$ of a point $x \in \gamma$ is bounded. Then we have that
	$H(\gamma) - H(\gamma \setminus \{x\}) = 6.$
\end{proposition}

\begin{figure}[h]
	\centering
	\includegraphics[width = 0.8\linewidth]{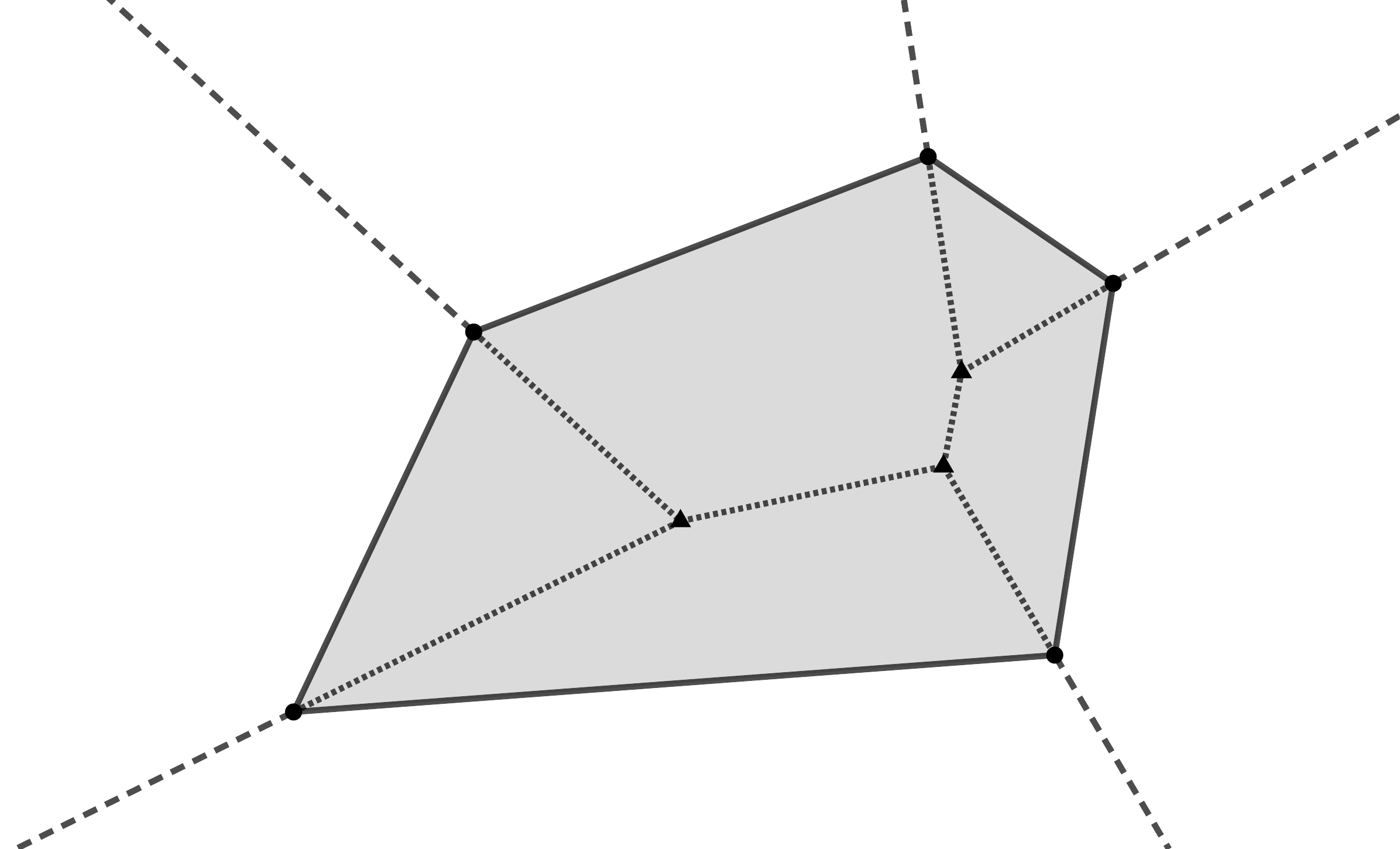}	
	\caption{Comparison of a Laguerre diagram with and without a point $x$. The Laguerre cell $L(x,\gamma)$ is the grey pentagon, full lines are its edges, the edges of the neighbouring cells in $L(\gamma)$ are the dashed lines, circular points are the vertices in $L(\gamma)$. Triangular points are the additional vertices in $L(\gamma \setminus \{x\})$ and dotted lines are the additional edges of the cells $L(y_i^x, \gamma \setminus \{x\})$ arising from the removal of the point~$x$.
	}\label{fig:DukazRozdilPoctuVrcholu}
\end{figure}

\begin{proof}
	Let $\gamma$ and $x$ be as assumed. Then $L(\gamma)$ (and also $L(\gamma \setminus {x})$ thanks to Lemma~\ref{lemma:zachovanipodminek}) is a generalized normal tessellation, and we know that $L(x, \gamma) = \bigcap_{i =1}^{k} P(x,y_i^x)$ for $y_i^x \in \gamma$ such that $L(x, \gamma) \cap L(y_i^x, \gamma) \neq \emptyset$, $k \in \N$. In particular, since $L(x,\gamma)$ is bounded, we have $\abs{\Delta_0(x,\gamma)} = k$.  The Laguerre cells of points $y \in \gamma \setminus \{y_1^x, \dots,y_{k}^x\}$ do not change by removing the point $x$, and therefore we can write
	\begin{equation}\label{vz:pom4}
	\begin{aligned}
		H(\gamma) - H(\gamma \setminus \{x\}) &= \abs{\Delta_0(x,\gamma)} + \sum_{i=1}^{k} \abs{\Delta_0(y_i^x,\gamma)} - \abs{\Delta_0(y_i^x, \gamma \setminus \{x\})} \\
		&= k + \sum_{i=1}^{k} \abs{\Delta_0(y_i^x,\gamma)} - \abs{\Delta_0(y_i^x, \gamma \setminus \{x\})}.
	\end{aligned}
\end{equation}
	By removing the point $x$, the neighbours of $x$ partition the cell $L(x,\gamma)$ into $k$ non-empty bounded convex polytopes $K_1, \dots, K_k$ such that $L(y_i^x, \gamma\setminus \{x\}) = K_i \cup L(y_i^x, \gamma)$.  Denote by $v_i$ the number of new vertices attained by the nucleus $y_i^x$ and realize that each neighbour $y_i^x$ shares two vertices with the nucleus $x$. Therefore
	\begin{equation}\label{vz:pom3}
		\begin{aligned}
			H(\gamma) - H(\gamma \setminus \{x\}) \overset{(\ref{vz:pom4})}{=}
			 k + \sum_{i=1}^{k} 2 - v_i = 3k - \sum_{i=1}^{k}v_i.
		\end{aligned}
	\end{equation}
	The partition of the cell $L(x,\gamma)$ by its neighbours defines a graph structure (see Figure~\ref{fig:DukazRozdilPoctuVrcholu}) with vertices $V = \Delta_0(x,\gamma) \cup V_2$, where $V_2$ is the set of new vertices, which appear after the removal of the point $x$, $V_2 = \Delta_0(\gamma \setminus \{x\}) \setminus \Delta_0(\gamma)$. The set of edges is defined as $E = \Delta_1(x,\gamma) \cup E_2$, where $E_2$ is the set of new edges (intersected with $L(x,\gamma)$), which appear after the removal of the point $x$. Since both $L(\gamma)$ and $L(\gamma\setminus \{x\})$ are normal, all vertices have degree 3. Thus, we have that
	\begin{equation} \label{vz:pom}
		3 \cdot \abs{V} = 2 \cdot \abs{E} \quad \implies \quad 3(k + \abs{V_2}) = 2(k + \abs{E_2}).
	\end{equation}
	Since we assume that there are no empty cells, the graph $(V,E_2)$ is a connected graph without cycles (i.\,e., a tree), and we know that
	\begin{equation} \label{vz:pom2}
		\abs{V} = \abs{E_2} + 1 \quad \implies \quad k + \abs{V_2} = \abs{E_2} + 1.
	\end{equation}
	Putting together (\ref{vz:pom}) and (\ref{vz:pom2}), we get that $\abs{V_2} = k- 2$. From normality we also get that $\sum_{i =1}^{k}v_i=3 \cdot \abs{V_2}$ and that, together with (\ref{vz:pom3}), completes the proof.
\end{proof}

\subsubsection{The existence of the limit measure and its support}

In what follows, we present the definitions and results directly taken from Section~3 in~\cite{ar:RoellyZass}. All of these results hold under the two assumptions $\hs$ and $\hmoment$ and the proofs can be found in \cite{ar:RoellyZass}.
Denote by $\dist{n} = \dist{\Lambda_n}$ the finite-volume Gibbs measure on $\Lambda_n = \left[-n,n\right)^2$, $n \in \N$. For $n \in \N$ and $\kappa \in \ztwo$ set $\Lambda^{\kappa}_n = \Lambda_n + 2n\kappa$. Then $\{\Lambda^{\kappa}_n\}_{\kappa \in \ztwo}$ is a disjoint partition of the space $\rtwo$.

For all $n \in \N$ let $\tildepn$ be the probability measure on $\Mset$ under which the configurations in disjoint sets $\Lambda^{\kappa}_n$ are independent and identically distributed according to the finite-volume Gibbs measure $\dist{n}$. For $\kappa \in \ztwo$ denote the shift operator on $\rtwo$ by $\vartheta_\kappa(x) = x + \kappa.$
Let $n \in \N$, $\Lambda \in \boreldvaom$ and define \textit{the empirical field $\barpn$ associated to the probability measure $\tildepn$} and \textit{the estimating sequence} $\hatpn$:
	\begin{equation} \label{df:stacionarniposl}
		\barpn = \frac{1}{\abs{\Lambda_n}} \sum_{\kappa \in \Lambda_n \cap \ztwo} \tildepn \circ \vartheta_\kappa^{-1}  \quad \text{ and } \quad \hatpn = \frac{1}{\abs{\Lambda_n}} \sum_{\kappa \in \ztwo \cap \Lambda_n: \Lambda \subset \vartheta_\kappa(\Lambda_n)} \dist{n} \circ \vartheta_\kappa^{-1}.
	\end{equation}

Denote by $\ProbMset$ the space of probability measures on $\Mset$. Function F on $\Mset$ is called  \textit{tame} if there exists $a>0$ such that $\left| F(\gamma) \right| \leq a\left(1 + \left<\gamma,1 + \norm{m}^ {2+\delta}\right>\right)$.
	Denote by $\mathcal{L}$ the set of all tame local functions $F:\Mset \goto \R$. We define \textit{the topology $\tau_\mathcal{L}$ of local convergence} on $\ProbMset$ as the smallest topology such that the mapping $\mathsf{P} \rightarrow \int F \drm \mathsf{P}$ is continuous for all $ F \in \mathcal{L}$.

\begin{lemma}[Proposition 1 in \cite{ar:RoellyZass}] \label{lemma:existenceOfBarP}
	Let $(\barpn)_{n\in \N}$ be the stationarised sequence defined in  (\ref{df:stacionarniposl}). Then there exists a subsequence $(\barp_{n_k})_{k\in \N}$ such that $\barp_{n_k} \overset{\tau_\mathcal{L}}{\goto} \barp$, where $\barp$ is a probability measure on $\Mset$ invariant under translations by $\kappa \in \ztwo$.
\end{lemma}

In the following text, we w.l.o.g.\ assume that $(\barpn)_{n\in \N}\overset{\tau_\mathcal{L}}{\goto} \barp$.
\begin{lemma}[Propositions 2 and 3 in \cite{ar:RoellyZass}]\label{lemma:TempAreSuppor}
	The measures $\barpn$ and $\barp$ defined in  (\ref{df:stacionarniposl}) satisfy that
	$\barpn(\Msettemp) = 1$, $n \in \N$,  and $\barp(\Msettemp) = 1.$ Furthermore, for all $ \varepsilon > 0$ there exists $\l \in \N$ such that
	$\barpn(\Msetl{\l}) \geq 1-\varepsilon, \forall n \in \N.$
\end{lemma}

\subsubsection{The set of admissible configurations}

We would like to show that for the energy function $H$ defined in (\ref{def:EFVertices}) the measure $\barp$ satisfies Definition \ref{def:infinitegibbsmeasure}.
First, we need to prepare some preliminary results. We will show that $\barp$-a.a.\ configurations satisfy that $L(\gamma)$ is a normal tessellation with no empty cells. We know, thanks to Lemmas \ref{lemma:TempAreSuppor} and \ref{lemma:podmr0r1protemp} (see Appendix), that $\barp$-a.a.$\,\gamma$ satisfy (R0) and (R1). Condition (R2) is satisfied, since $\barp$ is stationary under translations by $\kappa \in \ztwo$.

\begin{lemma} \label{lemma:podminkar2}
	If $\Psi$ is a simple marked point process whose distribution is invariant under translation by $\kappa \in \mathbb{Z}^2$ then it almost surely satisfies the assumption (R2) or it is empty, i.\,e., $\mathbb{P}(\text{conv}\{x': (x',x'') \in \Psi\} \in \{\rtwo, \emptyset\}) = 1$.
\end{lemma}

The proof of this lemma is just a slight modification of the proof of Theorem 2.4.4.\,in~\cite{bo:StochasticAndIntergralGeometry}. For assumptions (GP1) and (GP2) and the non-emptiness of the cells, we use the convergence in the $\tau_{\mathcal{L}}$ topology.

\begin{lemma} \label{lemma:podminkygp}
	It holds that $\barp$-a.a.$\,\gamma$ are in general position and satisfy $E(\gamma) = \emptyset$.
	
\end{lemma}
\begin{proof} Define for $k \in \N$ the sets $\Mset_{gp}^k = \{\gamma \in \Mset: \gamma_{\Lambda_k} \text{ is in general position}\}$ and $\Mset_{gp} = \{\gamma \in \Mset: \gamma \text{ is in general position}\}$.
	Then we have that $\Mset_{gp} = \bigcap_{k \in \N} \Mset_{gp}^k$ and also that $\Mset_{gp}^k \subset \Mset_{gp}^{k-1}$ and therefore any probability measure $\dist{}$ on $\Mset$ satisfies that $\lim_{k \to \infty} \dist{}(\Mset_{gp}^k)~=~\dist{}(\Mset_{gp}).$

	Now fix $k \in \N$.  Then according to (\ref{vz:PoissonGeneralPosition}) we have for all $\Lambda \in \boreldvaom$ and for all $z > 0$ that $\poislambdaz(\Mset_{gp}) = 1$, so also $\poislambdaz(\Mset_{gp}^k) = 1$. Therefore, for $n \geq k$ we have that
	$$\dist{n}(\Mset_{gp}^k) = \int_{\Mset_{gp}^k}\frac{1}{Z_{\Lambda_n}} e^{-H(\gamma_{\Lambda_n})} \poislambdanz(\drm \gamma)= 1$$
	and since $\Lambda_k \subset \Lambda_n$ we also have $\tildepn(\Mset_{gp}^k) = 1$.  Now take $\barpn
 = \frac{1}{(2n)^2}\sum_{\kappa \in \Lambda_n \cap \mathbb{Z}^2} \tildepn \circ \vartheta_\kappa^{-1}.
	$
	It holds that if $\Lambda_k + \kappa \subset \Lambda_n$ then $\tildepn \circ \vartheta_\kappa^{-1}(\Mset_{gp}^k) = 1$. It also holds that for all $ \kappa \in \Lambda_{n-k-1} \cap \mathbb{Z}^2$ we have $\Lambda_k + \kappa \subset \Lambda_n$, so we can write $\forall n \geq k + 1$
	\begin{align*}
		\barpn(\Mset_{gp}^k) &= \frac{1}{(2n)^2}\sum_{\kappa \in \Lambda_n \cap \mathbb{Z}^2} \tildepn\circ \vartheta_\kappa^{-1}(\Mset_{gp}^k) \\
		&= \frac{(2(n-k-1))^2}{(2n)^2} + \frac{1}{(2n)^2}\sum_{\kappa \in \Lambda_n \setminus \Lambda_{n-k-1} \cap \mathbb{Z}^2} \tildepn \circ \vartheta_\kappa^{-1}(\Mset_{gp}^k).
	\end{align*}
	Therefore, $\lim_{n\to\infty}\barpn(\Mset_{gp}^k)$ = 1. Since $\barp$ is a limit of  $\{\barpn\}_{n \in \N}$ in the $\tau_{\mathcal{L}}$ topology and $\ind \left[\gamma \in \Mset_{gp}^k \right]$ is a tame and local function, we get that $1=\lim_{n\to\infty}\barpn(\Mset_{gp}^k) = \barp(\Mset_{gp}^k).$
	This holds $\forall k \in \N$ and therefore $\barp(\Mset_{gp}) = 1$.
	For the second part, we define sets
	\begin{align*}
		\Mset_z = \{\gamma \in \Mset: E(\gamma) = \emptyset\}, \text{ and }
		\Mset_z^k = \{\gamma \in \Mset:  E(\gamma_{\Lambda_k}) = \emptyset\}, \, k \in \N, \, \Lambda_k = \left[-k,k \right)^2.
	\end{align*}
	Due to Lemma \ref{lemma:prazdnebunkjsoukonecnepruniky} and the fact that $\barp$-a.a.$\,\gamma \in \Mset$ satisfy the regularity conditions and are in general position, we can write
	$\lim_{k \to \infty} \barp(\Mset_{z}^k) = \barp(\Mset_{z}).$ Fix $k \in \N$,  then (\ref{vz:GibbsNonemptySets}) together with the implication $E(\gamma_{\Lambda_n}) = \emptyset \implies E(\gamma_{\Lambda_k}) = \emptyset$ imply that $\dist{n}(\Mset_z^k) = 1$ for all $n~\geq~k$.
	The rest of the proof follows analogously as in the previous case.
\end{proof}

\begin{definition}
	The set of \textit{admissible configurations} is defined as
	\begin{equation} \label{def:Mspruhem}
		\overline{\Mset} = \{\gamma \in \Msettemp: \gamma \text{ satisfies (R1), (R2), (GP1), (GP2) and } E(\gamma) = \emptyset\}\cup \{\zeromeasure\}.
	\end{equation}
\end{definition}

Using Lemmas \ref{lemma:podminkar2} and \ref{lemma:podminkygp} together with Lemma \ref{lemma:podmr0r1protemp} from Appendix, we have the following proposition.

\begin{proposition} \label{lemma:SupportOfBarP}
	Consider the set of admissible configurations $\overline{\Mset}$ defined in (\ref{def:Mspruhem}). It holds that $\barp(\overline{\Mset}) = 1$. Particularly for $\barp$-a.a.$\,\gamma \neq \zeromeasure$ we have that $L(\gamma)$ is a~normal tessellation.
\end{proposition}

\subsubsection{An infinite-volume Gibbs--Laguerre measure}

Recall formula (\ref{def: conditionalenergy}) for the conditional energy of a configuration $\gamma$ in $\Lambda$.
Thanks to Proposition \ref{lemma:KeyPropForEnergyFunctionVertices}, we know how this function looks for admissible configurations.
\begin{lemma} \label{lemma:CondEnForAdmissibleConf}
	Take $\gamma \in \overline{\Mset}$ and recall that our energy function is of the form (\ref{def:EFVertices}). Then we have that $H_\Lambda (\gamma) = 6 \cdot \abs{\gamma_\Lambda}$, $\forall \Lambda \in \boreldvaom$.
\end{lemma}

\begin{proof}
	If $\gamma = \zeromeasure$, then it clearly holds. For $\gamma \neq \zeromeasure$ we have that $\gamma_\Lambda = \{x_1, \dots, x_M\}$ for some $M \in \N$. Denote $\gamma_\Lambda^i = \{x_1, \dots, x_i\}$. From the definition of the conditional energy
	$$H_\Lambda(\gamma) = \lim_{n\to\infty} H(\gamma_{\Lambda_n}) - H(\gamma_{\Lambda_n \setminus \Lambda}) = \sum_{i=1}^{M} \lim_{n\to\infty} H(\gamma_{\Lambda_n \setminus \Lambda} \gamma_\Lambda^i) - H(\gamma_{\Lambda_n \setminus \Lambda}\gamma_\Lambda^{i-1}).$$
	Thanks to the assumptions on $\gamma$ we get that $L(\gamma)$ is a normal tessellation with no empty cells and therefore for all $i = 1, \dots, M$ there exists $n$ large enough so that $L(x_i, \gamma_{\Lambda_n \setminus \Lambda} \gamma_\Lambda^i)$ is bounded. Proposition \ref{lemma:KeyPropForEnergyFunctionVertices} implies $\lim_{n\to\infty} H(\gamma_{\Lambda_n \setminus \Lambda} \gamma_\Lambda^i) - H(\gamma_{\Lambda_n \setminus \Lambda}\gamma_\Lambda^{i-1}) = 6,$ which finishes the proof.
\end{proof}

\bigskip

Recall that $\Mset_a = \{\gamma \in \Mset: \supmark{\gamma} \leq a\}$, $a \in \N$, is the set of configurations whose marks are at most $a$. We define an increasing sequence of \textit{local sets} (i.\,e., subsets of $\Mset$  whose  indicator is a local function). Take $\Lambda \in \boreldvaom$ and $\l, n, a \in \N$ and define
\begin{equation}
	\begin{aligned}\label{def:MnyABC}
		&\C(\Lambda,a,\l,n) =  \left \lbrace \xi \in \Mset:  \xi \text{ satisfies  assumptions (C1) and (C2)}  \right \rbrace, \text{ where} \\
		& \qquad  \text{(C1)}: \, \text{there exists } u \in \xi_{\Lambda_n \setminus \Lambda}: u' \in U\left(0,\frac{1}{2}\l\right),   \\
		& \qquad \text{(C2)}: \, \forall \gamma \in \Mset_a, \,   \forall x \in \gamma_{\Lambda} \text{ we have } L(x, \xi_{\Lambda_n \setminus \Lambda} \cup \{x\}) \subset U\left(0, \frac{1}{2} \l\right). \\
	\end{aligned}
\end{equation}
Put
$\B(\Lambda,a,\l) = \bigcup_{n \in \N} \C(\Lambda,a,\l,n),$ and $\A(\Lambda,a) = \bigcup_{\l \in \N} \B(\Lambda,a,\l),$
then $\forall \Lambda \in \boreldvaom$ and  $\forall a,l,n \in \N$ it holds that
$\C(\Lambda,a,\l,n) \subset \C(\Lambda,a,\l+1,n), \, \C(\Lambda,a,\l,n) \subset \C(\Lambda,a,\l,n~+~1~),$ $\B(\Lambda,a,l) \subset \B(\Lambda,a,l+1), \text{ and } \A(\Lambda,a) \supset \A(\Lambda,a+1).$ We also have the following equality.

\begin{lemma} \label{lemma:OMnachABC}
	Take $\Lambda \in \boreldvaom$, then we have that
	\begin{equation*}
		\overline{\Mset} = \bigcap_{a \in \N} \bigcup_{\l \in \N} \bigcup_{n \in \N} \overline{\Mset} \cap \C(\Lambda,a,\l,n) \cup \{\zeromeasure\}.
	\end{equation*}
\end{lemma}

\begin{proof}
	The relation $\supset$ clearly holds. Take $\xi \in \overline{\Mset}$, $\xi \neq \zeromeasure$. We would like to show that $\forall a \in \N$ there exist $\l, n  \in \N$ such that $\xi \in \C(\Lambda,a,\l,n)$. Fix $a \in \N$ and consider
 $n \geq n_0 = \text{min} \{n \in \N: \exists u \in \xi_{\Lambda_n \setminus \Lambda}\}$ and $l \geq l_0 = \text{min} \{\l \in \N: \Lambda_{n_0} \subset U(0,\frac{1}{2}l)\}$. This will ensure that assumption (C1) is satisfied.

	Now \buno\ assume that $\Lambda$ is closed (otherwise, work with clo($\Lambda$)). We will use the observation that $L((x',x''), \gamma) \subset L((x',a), \gamma)$, whenever $x'' \leq a$. Therefore, to prove (C2), it is enough to prove that for some $n,l \in \N$ and $\forall x' \in \Lambda$ we have that $$L\left((x',a), \xi_{\Lambda_{n} \setminus \Lambda}\cup \{(x',a)\}\right) \subset U\left(0,\frac{1}{2}l\right).$$
	
	It holds (since $\xi \in \overline{\Mset}$) that $\forall x' \in \Lambda$ there exist $n_x, l_x$ such that
	$$
	L\left((x',a), \xi_{\Lambda_{n_x} \setminus \Lambda}\cup \{(x',a)\}\right) \subset U\left(0,\frac{1}{2}l_x\right).
	$$
	Then, because of the representation (\ref{def:LagCellAsIntersection}) and the openness of $U\left(0,\frac{1}{2}l_x\right)$, there exists $\eps_x> 0$ such that also $\forall y' \in U(x',\eps_x)$ we have that
	$$
	L\left((y',a), \xi_{\Lambda_{n_x} \setminus \Lambda}\cup \{(y',a)\}\right) \subset U\left(0,\frac{1}{2}l_x\right).
	$$
	Therefore, we have an open cover of $\Lambda$, $\Lambda \subset \bigcup_{x' \in \Lambda} U(x', \eps_x)$ and since $\Lambda$ is a compact set, there exists a finite cover $\Lambda \subset \bigcup_{i = 1}^N U(x_i', \eps_{x_i})$. To finish the proof, it is enough to take $n = \max \{n_0, n_{x_1}, \dots, n_{x_N}\}$ and $l = \max \{l_0, l_{x_1}, \dots, l_{x_N}\}$.
\end{proof}

\bigskip

Recall formula (\ref{df:GibbsianKernel}) for Gibbs kernel $
	\Xi_\Lambda(\xi,\drm \gamma) = \frac{e^{-H_\Lambda(\gamma_\Lambda\xi_{\Lambda^c})}}{Z_\Lambda(\xi)} \poislambdaz(\drm \gamma).$ We need to make sure that this quantity is well defined, at least for almost all configurations. To do that, we need the following observation which can be proven similarly as (\ref{vz:PoissonGeneralPosition}) in \cite{ar:LautensackZuyev}, Proposition 3.1.5. (see also \cite{bo:MollerRandomVoronoi}, Proposition 4.1.2.).

\begin{lemma} \label{lemma:asGPForBoundaryCondition}
	Take $\xi \in \Mset$ such that it is in general position. Then for all $\Lambda \in \boreldvaom$ and for all $z > 0$ we get that for $\poislambdaz$-a.a.\,$\gamma \in \Mset$ also $\xi_{\Lambda^c}\gamma_{\Lambda}$ is in general position.
\end{lemma}

Now we can show that the Gibbs kernel is well defined for all $\xi \in \overline{\Mset} \cup \Msetf$.

\begin{lemma} \label{lemma:ZLambdaDobreDef}
	Let $\xi \in \overline{\Mset}$ or $\xi \in \Msetf$ such that it is in general position and $E(\xi) = \emptyset$, then we have that $0 < Z_\Lambda(\xi) < \infty$, $\forall \Lambda \in \boreldvaom,\, \forall z > 0$.
\end{lemma}

\begin{proof}
	At first take $\xi \in \overline{\Mset}$. Then we know that $\xi$ is in general position, satisfies the regularity conditions, and also $E(\xi) = \emptyset$. Thanks to Lemma \ref{lemma:zachovanipodminek} we have that also $\xi_{\Lambda^c}\gamma_{\Lambda}$, $ \gamma \in \Mset$, satisfies the regularity conditions and according to Lemma~\ref{lemma:asGPForBoundaryCondition} we have that for $\poislambdaz$-a.a.$\,\gamma \in \Mset$ it holds that $\xi_{\Lambda^c}\gamma_{\Lambda}$ is in general position. If $E(\xi_{\Lambda^c}\gamma_{\Lambda}) \neq \emptyset$, then $H_\Lambda(\xi_{\Lambda^c}\gamma_{\Lambda}) = + \infty$. Otherwise, thanks to Lemma~\ref{lemma:CondEnForAdmissibleConf} we get that $H_\Lambda(\xi_{\Lambda^c}\gamma_{\Lambda}) = 6 \abs{\gamma_{\Lambda}}$. Altogether $H_\Lambda(\xi_{\Lambda^c}\gamma_{\Lambda}) \geq 0$ for $\poislambdaz$-a.a.$\, \gamma$, and hence $Z_\Lambda(\xi) = \int e^{-H_\Lambda(\gamma_\Lambda \xi_{\Lambda^c})} \poislambdaz(\drm \gamma) < \infty.$
	Now take $\xi \in \Msetf$, which is in general position and has no empty cells, and denote $M = \abs{\xi_{\Lambda^c}}$. Then we can write
	$		H_\Lambda(\xi_{\Lambda^c}\gamma_{\Lambda}) = H(\xi_{\Lambda^c}\gamma_{\Lambda}) - H(\xi_{\Lambda^c})
		\geq - H(\xi_{\Lambda^c}) \geq - 3 \cdot {M \choose 3},
$
	since $L(\xi_{\Lambda^c})$ can have at most ${M \choose 3}$ vertices. Hence
	$Z_\Lambda(\xi) = \int e^{-H_\Lambda(\gamma_\Lambda \xi_{\Lambda^c})} \poislambdaz(\drm \gamma) < \infty.$
	
\end{proof}

\bigskip

Particularly, $\Xi_\Lambda$ (recall formula (\ref{df:GibbsianKernel})) is well defined for all $\xi \in \overline{\Mset} $ and $\xi \in \Msetf$ which are in general position and satisfy $E(\xi) = \emptyset$. Define \textit{the cut-off kernel}
\begin{equation*}
	\Xi_\Lambda^{n,a}(\xi,\drm \gamma) = \frac{\ind \{\gamma_\Lambda \in \Mset_a\} \cdot  e^{-H_\Lambda(\gamma_\Lambda\xi_{\Lambda_n\setminus \Lambda})}}{Z^{n,a}_\Lambda(\xi_{\Lambda_n\setminus \Lambda})}  \poislambdaz(\drm \gamma).
\end{equation*}
Using the second part of the proof of Lemma \ref{lemma:ZLambdaDobreDef}, we can see that $\Xi_\Lambda^{n,a}$ is well defined for all $\xi$ in general position with $E(\xi) = \emptyset$.

Recall sets $\Msetl{l}$ from Section \ref{subsec:TCaGM}. The following final auxiliary lemma justifies the definition (\ref{def:MnyABC}) of the sets $\C(\Lambda,a,l,n)$. These sets are chosen so that the conditional energy depends only on the boundary condition inside $\Lambda_n$.
\begin{lemma} \label{lemma:RovnostPodmEnergie}
	Let $\Lambda \in \boreldvaom$, and take $a, n,l \in \N$ such that $U\left(0,2l+1\right) \subset \Lambda_n$ and $\Lambda \oplus B(0,a) \subset U(0,\frac{1}{2}l)$. Then for all $\xi \in \C(\Lambda,a,l,n)\cap \Msetl{l}$ and for all $ \gamma \in \Mset_a$ such that $\xi_{\Lambda^c}\gamma_\Lambda$ are in general position and $E(\xi_{\Lambda^c}) = \emptyset$ we have that
	\begin{itemize} \label{vz:EkvPrazdnychMnozin}
		\item[i)]$E(\xi_{\Lambda^c}\gamma_\Lambda) \neq \emptyset \iff  E(\xi_{\Lambda_n \setminus \Lambda}\gamma_\Lambda) \neq \emptyset,$
		\item[ii)] $ H_\Lambda (\xi_{\Lambda^c}\gamma_\Lambda) =  H_\Lambda (\xi_{\Lambda_n \setminus \Lambda}\gamma_\Lambda)$.
	\end{itemize}
\end{lemma}

\begin{proof}
First, we assume i) and prove ii).

Take $\xi, \gamma$ satisfying the assumptions. We have $\gamma_\Lambda = \{x_1, \dots, x_M\}$ for some $M \in \N$. Denote by $\gamma_\Lambda^i = \{x_1, \dots, x_i\}$, $i=1,\dots,M$. If $E(\xi_{\Lambda^c}\gamma_\Lambda) \neq \emptyset$ then according to i) also $E(\xi_{\Lambda_n \setminus \Lambda}\gamma_\Lambda) \neq \emptyset$ and we have $H_\Lambda (\xi_{\Lambda^c}\gamma_\Lambda) = + \infty = H_\Lambda (\xi_{\Lambda_n \setminus \Lambda}\gamma_\Lambda).$

	If $E(\xi_{\Lambda^c}\gamma_\Lambda) = E(\xi_{\Lambda_n \setminus \Lambda}\gamma_\Lambda) = \emptyset$, then thanks to the definition of the set $\C(\Lambda,a,l,n)$ we have that the cells $L(x_i,\xi_{\Lambda_n\setminus \Lambda}\gamma_\Lambda^i)$ are bounded $\forall i \in \{1, \dots,M\}$. Recalling Proposition~\ref{lemma:KeyPropForEnergyFunctionVertices} for our energy function $H$, we can write
	\begin{align*}
		H_\Lambda(\xi_{\Lambda_n\setminus \Lambda}\gamma_\Lambda) &= H(\xi_{\Lambda_n\setminus \Lambda}\gamma_\Lambda) - H(\xi_{\Lambda_n\setminus \Lambda}) \\&= \sum_{i=1}^{M} H(\xi_{\Lambda_n\setminus \Lambda}\gamma_\Lambda^i) - H(\xi_{\Lambda_n\setminus \Lambda}\gamma_\Lambda^{i-1})
		\overset{\text{P.} \ref{lemma:KeyPropForEnergyFunctionVertices}}{=} \sum_{i=1}^{M} 6 = 6 \cdot \abs{\gamma_\Lambda}.
	\end{align*}
	Using Lemma \ref{lemma:CondEnForAdmissibleConf} we also have that $H_\Lambda (\xi_{\Lambda^c}\gamma_\Lambda) = 6 \cdot \abs{\gamma_\Lambda}$.

	Now, it remains to prove i). Take $\xi, \gamma$ satisfying the assumptions. The implication $\impliedby$ always holds, so we only have to prove that if there exists an empty cell for $\xi_{\Lambda^c}\gamma_\Lambda$, then it is already empty in $\xi_{\Lambda_n\setminus\Lambda}\gamma_\Lambda$ (remember that $E(\xi_{\Lambda^c}) = \emptyset$).

	Let there exist $x \in \xi_{\Lambda^c}\gamma_\Lambda$ such that $L(x,\xi_{\Lambda^c}\gamma_\Lambda) = \emptyset$ and assume for contradiction that $E(\xi_{\Lambda_n\setminus\Lambda}\gamma_\Lambda) = \emptyset$. This means that either $x \in \xi_{\Lambda_n^c}$ or $L(x,\xi_{\Lambda_n\setminus\Lambda}\gamma_\Lambda) \neq \emptyset$. Recall  Lemma \ref{lemma:pomocne} and consider the three possible locations of the point $x$:

	1) $x \in \gamma_\Lambda$: Then $\forall z \in L(x,\xi_{\Lambda_n\setminus\Lambda}\gamma_\Lambda)$ there exists $y \in \xi_{\Lambda_n^c}$ such that $\rho(z,y) \leq \rho(z,x)$. However, from the choice of $n$ and $l$ and from the definition of the set $\C(\Lambda,a,l,n)$ we know that $y'\in U(0,2l+1)^c$, $x' \in U(0,\frac{1}{2}l)$ and $z \in U(0,\frac{1}{2}l)$. Using Lemma \ref{lemma:pomocne} we get that
	\begin{equation*}
		\rho(z,y) \leq \rho(z,x) \leq l^2 \overset{\text{L.} \ref{lemma:pomocne}}{<} \rho(z,(y',\abs{y'}-l)) \leq \rho(z,y),
	\end{equation*}
	which is clearly a contradiction.

	2) $x \in \xi_{\Lambda_n\setminus\Lambda}$: We know that $L(x,\xi_{\Lambda^c}\gamma_\Lambda) = \emptyset$ but
	$L(x,\xi_{\Lambda_n\setminus\Lambda}\gamma_\Lambda) \neq \emptyset$ and also $L(x,\xi_{\Lambda^c}) \neq \emptyset.$ Therefore,
	\begin{equation} \label{vz:pompom}
		\begin{aligned}
			\forall z \in L(x,\xi_{\Lambda^c}) \, \exists u \in \gamma_\Lambda &\text{ such that } z \in L(u,\xi_{\Lambda^c}\gamma_\Lambda), \\
			\forall z \in L(x,\xi_{\Lambda_n\setminus\Lambda}\gamma_\Lambda) \, \exists y \in \xi_{\Lambda_n^c} &\text{ such that } z \in L(y,\xi_{\Lambda^c}\gamma_\Lambda).
		\end{aligned}
	\end{equation}
	If $\exists z \in L(x,\xi_{\Lambda^c})\cap L(x,\xi_{\Lambda_n\setminus\Lambda}\gamma_\Lambda)$, then there exist $u \in \gamma_\Lambda$ and $y \in \xi_{\Lambda_n^c}$ such that $\rho(z,y) = \rho(z,u)$ and again we get a contradiction with Lemma \ref{lemma:pomocne}.
	
	Therefore $ L(x,\xi_{\Lambda^c})\cap L(x,\xi_{\Lambda_n\setminus\Lambda}\gamma_\Lambda) = \emptyset$. Then there exists $z \in L(x,\xi_{\Lambda_n\setminus\Lambda}\gamma_\Lambda)$ such that $\exists u \in \gamma_\Lambda$ such that $\rho(z,x) = \rho(z,u)$, i.\,e., $z \in L(u,\xi_{\Lambda_n\setminus\Lambda}\gamma_\Lambda)$. Since by (\ref{vz:pompom}) there also exists $y \in \xi_{\Lambda_n^c}$ such that $z \in L(y,\xi_{\Lambda^c}\gamma_\Lambda)$, we again get the contradiction $\rho(z,y) \leq \rho(z,u)$.

	3) $x \in \xi_{\Lambda_n^c}$: We know that $L(x,\xi_{\Lambda^c}\gamma_\Lambda) = \emptyset$ and $L(x,\xi_{\Lambda^c}) \neq \emptyset$. Therefore for all $ z \in L(x,\xi_{\Lambda^c})$ there exists $u \in \gamma_\Lambda$ such that $\rho(z,u) < \rho(z,x)$. Particularly, we can assume that $\rho(z,u) \leq \rho(z,v)$ for all $v \in \gamma_\Lambda$ and therefore $z \in L(u,\xi_{\Lambda^c}\gamma_\Lambda) \subset U(0,\frac{1}{2}l)$. Therefore $L(x,\xi_{\Lambda^c}) \subset U(0,\frac{1}{2}l)$.  Notice that $x' \in U(0,2l+1)^c$. From the definition of the set $\C$ there exists $y \in \xi_{\Lambda_n\setminus \Lambda}$ such that $y' \in U(0,\frac{1}{2}l)$, which implies that $\forall z \in L(x,\xi_{\Lambda^c})$ we have that
	\begin{equation*}
		\rho(z,x) \leq \rho(z,y) \leq l^2 \overset{\text{L.} \ref{lemma:pomocne}}{<} \rho(z,(x',\abs{x'}-l)) \leq \rho(z,x),
	\end{equation*}
	which is the final contradiction and the proof is finished.
\end{proof}

\bigskip

We are now ready to prove our main result.

\begin{theorem} \label{th:DLRproMozaiky}
	Consider the probability measure $\barp$ from Lemma \ref{lemma:existenceOfBarP} and assume that it satisfies $\barp(\{\zeromeasure\}) = 0$. Then the $\DLR_\Lambda$ equations hold for all $\Lambda \in \boreldvaom$ and for all measurable bounded local functions $F$. Particularly $\barp$ is an infinite-volume Gibbs measure with energy function $H$ defined in (\ref{def:EFVertices}) and activity $z>0$.
\end{theorem}

\begin{proof}
	Take $\Lambda \in \boreldvaom$, measurable bounded $\Lambda$-local function $F$, We will show that
	\begin{equation*}
		\delta_0 = \abs{\int F(\gamma) \barp(\drm \gamma) - \int \int F(\gamma_\Lambda)\Xi_\Lambda(\xi, \drm \gamma) \barp(\drm \xi)} < \eps,\,  \forall \eps > 0.
	\end{equation*}
	Fix $\eps > 0$. Find $i_0$ smallest such that $\Lambda \subset \Lambda_{i_0}$. We will \buno\ assume that $i_0=1$ (otherwise work with $n \geq i_0$ in the whole proof). Then there exists $a \in \N$ such that
	\begin{itemize}
		\item[1.] $\poislambdaz(\Mset_a) \geq 1-\eps$.
	\end{itemize} For this $a$ find $l \in \N$ such that
	\begin{itemize}
		\item[2.] $\Lambda \oplus B(0,a) \subset U(0,\frac{1}{2}l)$,
		\item[3.] $\barp(\Msetl{l}) \geq 1-\eps$, $\barpn(\Msetl{l}) \geq 1- \eps$ for all $n \in \N$ (from Lemma \ref{lemma:TempAreSuppor}),
		\item[4.] $\barp(\B(\Lambda,a,l)) \geq 1-\eps$ (from Proposition \ref{lemma:SupportOfBarP}, Lemma \ref{lemma:OMnachABC} and  $\barp(\{\zeromeasure\}) = 0$).
	\end{itemize}
	For these $a$ and $l$ we can find $k \in \N$ such that
	\begin{itemize}
		\item[5.] $U(0,2l+1), \subset \Lambda_k$,
		\item[6.] $\barp(\C(\Lambda,a,l,k)) \geq 1-2\eps$.
	\end{itemize}
	Fix $a,l,k$ and recall measures $\hatpn$ defined in (\ref{df:stacionarniposl}). It holds that $\hatpn$ satisfy $(\DLR)_\Lambda$ and they are asymptotically equivalent to $\barpn$ in the sense that for any $G \in \mathcal{L}$ we get that $ \lim_{n\to\infty} \abs{\int G(\gamma)\hatpn(\drm \gamma) - \int G(\gamma) \barpn (\drm \gamma)} = 0$ (see \cite{ar:RoellyZass}, page 988).

	In particular, there exists $n_0$ such that $\forall n \geq n_0$ we get that $\hatpn(\Mset) \geq 1-\eps$. It also holds that $\barpn((\Msetl{l})^c) \geq \hatpn((\Msetl{l})^c)$. Therefore there exists $n_1 \geq n_0$ such that
	\begin{itemize}
		\item[7.] $\hatpn(\Msetl{l}) \geq 1-2\eps$ for all $n \geq n_1$,
		\item[8.] $\hatpn(\C(\Lambda,a,l,k)) \geq 1-3\eps$ for all $n \geq n_1$.
	\end{itemize}
	Now we have everything we need to estimate $\delta_0$. Assume \buno\ that $\abs{F} \leq 1$ and recall that $\barp(\overline{\Mset}) = 1$.
	\begin{align*}
		&\delta_0 = \abs{\int F(\gamma) \barp(\drm \gamma) - \int \int F(\gamma_\Lambda)\Xi_\Lambda(\xi, \drm \gamma) \barp(\drm \xi)}  \leq\barp((\C(\Lambda,a,l,k) \cap \Msetl{l})^c) \\
		&\qquad \qquad \qquad \qquad + \abs{\int F(\gamma) \barp(\drm \gamma) - \int_{\C(\Lambda,a,l,k) \cap \Msetl{l}} \int F(\gamma_\Lambda)\Xi_\Lambda(\xi, \drm \gamma) \barp(\drm \xi)} \\
		&\overset{3.,6.}{\leq} 3\eps + \abs{\int F(\gamma) \barp(\drm \gamma) - \int_{\C(\Lambda,a,l,k) \cap \Msetl{l}} \int F(\gamma_\Lambda)\Xi_\Lambda^{k,a}(\xi, \drm \gamma) \barp(\drm \xi)} \\
		& \qquad \quad  + \abs{\int_{\C(\Lambda,a,l,k) \cap \Msetl{l}} \left[\int_{\Mset_a} F(\gamma_\Lambda)\Xi_\Lambda^{k,a}(\xi, \drm \gamma)  -  \int_{\Mset_a} F(\gamma_\Lambda)\Xi_\Lambda(\xi, \drm \gamma)\right] \barp(\drm \xi)} \\
		&  \qquad  \quad + \abs{\int_{\C(\Lambda,a,l,k) \cap \Msetl{l}} \int_{(\Mset_a)^c} F(\gamma_\Lambda)\Xi_\Lambda(\xi, \drm \gamma) \barp(\drm \xi)}.
	\end{align*}
	Now we have for some $b < \infty$:
	\begin{equation}
		\begin{aligned} \label{vz:odhad1}
			&\abs{\int_{\C(\Lambda,a,l,k) \cap \Msetl{l}} \int_{(\Mset_a)^c} F(\gamma_\Lambda)\Xi_\Lambda(\xi, \drm \gamma) \barp(\drm \xi)} \\
			& \leq \int \int_{(\Mset_a)^c} \frac{1}{Z_\Lambda(\xi)} \poislambdaz(\drm \gamma) \barp(\drm \xi) \leq \poislambdaz((\Mset_a)^c) \cdot \frac{1}{\poislambdaz(\{\zeromeasure\})} \overset{1.}{\leq} b \cdot \eps.
		\end{aligned}
	\end{equation}
	Now for $\barp$-a.a. $\xi \in \C(\Lambda,a,l,k) \cap \Msetl{l}$ we can use Lemma \ref{lemma:RovnostPodmEnergie} to show that
	\begin{equation}
		\begin{aligned} \label{vz:odhad2}
			&\abs{\int_{\Mset_a} F(\gamma_\Lambda)\Xi_\Lambda^{k,a}(\xi, \drm \gamma)  -  \int_{\Mset_a} F(\gamma_\Lambda)\Xi_\Lambda(\xi, \drm \gamma)}\\
			& = \abs{\int_{\Mset_a} F(\gamma_\Lambda) e^{-H_\Lambda(\xi_{\Lambda^c}\gamma_\Lambda)} \left(\frac{Z_\Lambda(\xi) - Z_\Lambda^{k,a}(\xi)}{Z_\Lambda(\xi) \cdot Z_\Lambda^{k,a}(\xi)}\right)\poislambdaz(\drm \gamma)} \\
			&\leq \int_{\Mset_a} \frac{\abs{Z_\Lambda(\xi) - Z_\Lambda^{k,a}(\xi)}}{\poislambdaz(\{\zeromeasure\})^2}\poislambdaz(\drm \gamma) \leq b^2 \cdot \abs{Z_\Lambda(\xi) - Z_\Lambda^{k,a}(\xi)} 
			\overset{1.}{\leq} b^2 \cdot \eps.
		\end{aligned}
	\end{equation}
	Therefore, we can estimate
	\begin{align*}
		\delta_0 \leq c \cdot \eps + \abs{\int F(\gamma) \barp(\drm \gamma) - \int_{\C(\Lambda,a,l,k) \cap \Msetl{l}} \int F(\gamma_\Lambda)\Xi_\Lambda^{k,a}(\xi, \drm \gamma) \barp(\drm \xi)} =: c \cdot \eps + \delta_1,
	\end{align*}
	where $c = 3 + b + b^2$. We continue with $\delta_1$:
	\begin{align*}
		\delta_1 
		&\leq \barp((\C(\Lambda,a,l,k) \cap \Msetl{l})^c) +\abs{\int F(\gamma) \barp(\drm \gamma) - \int \int F(\gamma_\Lambda)\Xi_\Lambda^{k,a}(\xi, \drm \gamma) \barp(\drm \xi)} \\
		& \overset{3.,6.}{\leq} 3 \eps + \abs{\int F(\gamma) \barp(\drm \gamma) - \int \int F(\gamma_\Lambda)\Xi_\Lambda^{k,a}(\xi, \drm \gamma) \barp(\drm \xi)} = : 3\eps + \delta_2.
	\end{align*}
	Now we use the asymptotic equivalence for $\hatpn$ and $\barpn$ and the fact that $F(\gamma)$ and $G(\gamma) = \int F(\nu) \Xi_\Lambda^{k,a}(\gamma, \drm \nu)$ are bounded (and therefore tame) local functions. Let $n \geq n_1$, then we have the following estimate for $\delta_2$:
	\begin{align*}
		\delta_2 &= \abs{\int F(\gamma) \barp(\drm \gamma) - \int \int F(\gamma_\Lambda)\Xi_\Lambda^{k,a}(\xi, \drm \gamma) \barp(\drm \xi)} \\ &\leq\abs{\int F(\gamma) \drm\barp - \int F(\gamma) \drm\hatpn} + \abs{\int F(\gamma) \hatpn(\drm \gamma) - \int \int F(\gamma_\Lambda)\Xi_\Lambda^{k,a}(\xi, \drm \gamma) \hatpn(\drm \xi)} \\
		& + \abs{\int \int F(\gamma_\Lambda)\Xi_\Lambda^{k,a}(\xi, \drm \gamma) \barp(\drm \xi) - \int \int F(\gamma_\Lambda)\Xi_\Lambda^{k,a}(\xi, \drm \gamma) \hatpn(\drm \xi)}.
	\end{align*}
	We can choose $n_2 \geq n_1$ so that $\forall n \geq n_2$ we have
	\begin{align*}
		&\abs{\int F(\gamma) \barp(\drm \gamma) - \int F(\gamma) \hatpn(\drm \gamma)}  \leq \eps,\\
		&\abs{\int \int F(\gamma_\Lambda)\Xi_\Lambda^{k,a}(\xi, \drm \gamma) \barp(\drm \xi) - \int \int F(\gamma_\Lambda)\Xi_\Lambda^{k,a}(\xi, \drm \gamma) \hatpn(\drm \xi)} \leq \eps.
	\end{align*}
	Therefore, for $n \geq n_2$ we can write
	\begin{align*}
		\delta_2 
		& \leq 2 \eps  + \abs{\int F(\gamma) \hatpn(\drm \gamma) - \int \int F(\gamma_\Lambda)\Xi_\Lambda^{k,a}(\xi, \drm \gamma) \hatpn(\drm \xi)} =: 2\eps + \delta_3.
	\end{align*}
	Now for our last estimate. Since $\hatpn$ satisfies $\DLR_\Lambda$, we can write
	\begin{align*}
		\delta_3 &= \abs{\int F(\gamma) \hatpn(\drm \gamma) - \int \int F(\gamma_\Lambda)\Xi_\Lambda^{k,a}(\xi, \drm \gamma) \hatpn(\drm \xi)} \\
		&\leq 0 + \abs{\int \int F(\gamma_\Lambda)\Xi_\Lambda(\xi, \drm \gamma) \hatpn(\drm \xi) - \int \int F(\gamma_\Lambda)\Xi_\Lambda^{k,a}(\xi, \drm \gamma) \hatpn(\drm \xi)} \\
		&\leq \abs{\int_{\C(\Lambda,a,l,k) \cap \Msetl{l}} \left[ \int F(\gamma_\Lambda)\Xi_\Lambda(\xi, \drm \gamma) - \int F(\gamma_\Lambda) \Xi_\Lambda^{k,a}(\xi, \drm \gamma) \right] \hatpn(\drm \xi)} \\
		& \qquad \qquad \qquad \qquad \qquad + 2\cdot \hatpn((\C(\Lambda,a,l,k) \cap \Msetl{l})^c) \\
		&\overset{7.,8.}{\leq} \abs{\int_{\C(\Lambda,a,l,k) \cap \Msetl{l}} \left[ \int_{\Mset_a} F(\gamma_\Lambda)\Xi_\Lambda(\xi, \drm \gamma) - \int_{\Mset_a} F(\gamma_\Lambda) \Xi_\Lambda^{k,a}(\xi, \drm \gamma) \right] \hatpn(\drm \xi)} \\
		& \qquad \qquad \qquad \qquad \qquad +\abs{\int \int_{(\Mset_a)^c} F(\gamma_\Lambda)\Xi_\Lambda(\xi, \drm \gamma) \hatpn(\drm \xi)} + 10 \cdot \eps.
	\end{align*}
	Now analogously as in (\ref{vz:odhad1}) and (\ref{vz:odhad2}) we can estimate
	\begin{align*}
		&\abs{\int \int_{(\Mset_a)^c} F(\gamma_\Lambda)\Xi_\Lambda(\xi, \drm \gamma) \hatpn(\drm \xi)} \leq b \cdot \eps \\
		&\abs{\int_{\C(\Lambda,a,l,k) \cap \Msetl{l}} \int_{\Mset_a} F(\gamma_\Lambda)\left[\Xi_\Lambda(\xi, \drm \gamma) - \Xi_\Lambda^{k,a}(\xi, \drm \gamma) \right] \hatpn(\drm \xi)} \leq b^2 \cdot \eps.
	\end{align*}
	Putting everything together, we get that (recall that $c = 3 + b + b^2$)
	\begin{equation*}
		\delta_0 \leq c \cdot \eps + \delta_1 \leq (c + 3)\eps + \delta_2 \leq (c + 5)\eps + \delta_3 \leq (2c + 12) \eps.
	\end{equation*}
	This finishes the proof.
\end{proof}

\section{Concluding remarks}
In this paper, we have commented on the recent existence result from \cite{ar:RoellyZass} for the Gibbs marked point processes. We
provided a new formulation of the range assumption and applied the
general theorem to the family of Gibbs facet processes with repulsive interactions in $\rd$, $d \geq 2$. We also proved that the finite-volume Gibbs facet processes with attractive interactions do not exist in $\rtwo$. We believe that it should be possible to show that the finite-volume Gibbs facet processes with attractive interactions do not exist in any dimension. In the last section, we considered the Gibbs--Laguerre tessellations of $\rtwo$ and proved that, under the assumption that $\barp(\{\zeromeasure\}) = 0$, the infinite-volume Gibbs--Laguerre process exists for the energy function given in (\ref{def:EFVertices}).

It is natural to ask, whether the result of the last section can be generalized to random tessellations of $\rtwo$ with some other energy function or to higher dimensions (we thank the referee for raising these questions). Regarding the first one (some other energy function in $\rtwo$), the key ingredients here are the definition of the sets $\C(\Lambda,a,l,n)$ in (\ref{def:MnyABC}) and Proposition \ref{lemma:KeyPropForEnergyFunctionVertices}, which lead to Lemma \ref{lemma:RovnostPodmEnergie}. Therefore we would need for the other energy function to either satisfy that $H(\gamma)-H(\gamma \setminus \{x\})$ is equal to a constant (which does not seem probable) or we would need an analogy of Lemma \ref{lemma:RovnostPodmEnergie} to be true. At this moment, we are not aware of any such energy function.  Regarding the generalization to higher dimensions for energy function given by  (\ref{def:EFVertices}), both \cite{ar:RoellyZass} and \cite{ar:LautensackZuyev} are formulated for general $d \geq 2$ and it is true that it should be possible to formulate several of the results of Section \ref{subsec:Tessellations} for $d \geq 2$. However the proof of Proposition \ref{lemma:KeyPropForEnergyFunctionVertices} does not work for higher dimensions, since cycles appear in the graph structure even if we forbid empty cells and therefore the result for trees cannot be used. Considering generalizations to other energy functions in higher dimensions, the same remarks as for $d=2$ apply.

\section{Appendix}

This appendix contains several technical lemmas about Laguerre diagram and its cells and the connection between tempered configurations and Laguerre theory. These results are used in Section \ref{Sec:Gibbs-LaguerreMeasures}.
Although we formulate them for $d=2$, since the approach of Section \ref{Sec:Gibbs-LaguerreMeasures} cannot be easily generalised to higher dimensions, it should be possible to reformulate the results in this section for $d \geq 2$ (provided we consider the general formulation of conditions (R1), (R2), (GP1) and (GP2) from \cite{ar:LautensackZuyev}).

Let $\gamma \subset \rtwo \times (0,\infty)$, then the following lemma can be easily shown.

\begin{lemma}\label{lemma:zachovanipodminek}
	If $\gamma$ satisfies (R1), (R2), (GP1) and (GP2) and $E(\gamma) = \emptyset$. Then, for all $x \in \gamma$ also $E(\gamma \setminus x) = \emptyset$ and $\gamma \setminus x$ satisfies (R1), (R2), (GP1) and (GP2).
\end{lemma}

It holds that Laguerre cells can be represented as a finite intersection of the closed half-planes $P(x,y)$. The proof of the following claim is just a slight modification of the proof of Lemma 10.1.1. in \cite{bo:StochasticAndIntergralGeometry}.

\begin{lemma} \label{lemma:neprazdnebunkyjsoukonecnepruniky}
	Let $\gamma \subset \rtwo \times (0, \infty)$ be such that $L(\gamma)$ is a tessellation. Then for all $L(x,\gamma) \in L(\gamma)$ there exist $k_x \in \N$ and $y_i^x \in \gamma \setminus E(\gamma)$, $i = 1, \dots, k_x$, such that
	\begin{equation} \label{vztah}
		L(x,\gamma) = \bigcap_{i =1}^{k_x} P(x,y_i^x).
	\end{equation}
\end{lemma}

The points $\{y_1^x, \dots,  y_{k_x}^x\}$ can be chosen as those points whose cells intersect the cell $L(x,\gamma)$ and we call them \textit{neighbours} of the point $x$. The same holds if $\gamma$ is finite and in general position.
If we take into consideration the definition of tempered configurations, we can get a similar result for empty Laguerre cells.

\begin{lemma}\label{lemma:prazdnebunkjsoukonecnepruniky}
	Let $\gamma \subset \rtwo \times (0,\infty)$ be such that $\gamma \in \Msettemp$, it satisfies the regularity conditions and is in general position. Then $\forall x \in E(\gamma)$ there exist $k_x \in \N$ and $y_{1}^x, \dots, y_{k_x}^x \in \gamma$ such that $L(x,\gamma) = \bigcap_{i =1}^{k_x} P(x,y_i^x)$.
\end{lemma}

\begin{proof}
	Let $x \in E(\gamma)$.
	Then it must hold that $B(x',x'') \subset \bigcup_{y \in \gamma, y \neq x} B(y',y'').$ The set $B(x',x'')$ is bounded and $\gamma$ is tempered, therefore, there exists $l$ such that $l \geq l(t)$, where $l(t)$ is from  (\ref{lemma:vlasttemp}), and $B(x',x'') \subset U(0,l)$. Therefore, we know that $\forall y \in \gamma_{U(0,2l+1)^c}$ we have that $B(x',x'') \cap B(y',y'') = \emptyset$. So we can write
	$$B(x',x'') \subset \bigcup_{y \in \gamma_{U(0,2l+1)},\,  y \neq x} B(y',y'').$$
	Let $\varphi = \gamma_{U(0,2l+1)^c} \cup \{x\}$, then according to Lemma \ref{lemma:zachovanipodminek}  it holds that $\varphi$ satisfies the regularity conditions and is in general position. Particularly $L(\varphi)$ is a~(normal) tessellation. Furthermore it holds that $B(x',x'') \not\subset \bigcup_{y \in \varphi, y \neq x} B(y',y'')$ and therefore $ \emptyset \neq L(x,\varphi) \in L(\varphi)$. This allows us to use Lemma \ref{lemma:neprazdnebunkyjsoukonecnepruniky} and we get that there exist $y_{1}^x, \dots, y_{n_x}^x \in \varphi$ such that $L(x,\varphi) = \bigcap_{i =1}^{n_x} P(x,y_i^x)$. Altogether we get that
	\begin{align*}
		L(x,\gamma) &= \bigcap_{y \in \gamma} P(x,y) = \bigcap_{ y \in \gamma_{U(0,2l+1)}} P(x,y) \cap \bigcap_{y \in \gamma_{U(0,2l+1)^c}} P(x,y) = \\
		& = \bigcap_{ y \in \gamma_{U(0,2l+1)}} P(x,y) \cap  L(x,\varphi) = \bigcap_{ y \in \gamma_{U(0,2l+1)}} P(x,y) \cap \bigcap_{i =1}^{n_x} P(x,y_i^x).
	\end{align*}
	There are only finitely many points in $\gamma_{U(0,2l+1)}$, which completes the proof.
\end{proof}

\bigskip

What follows now is an auxiliary lemma for the proof that tempered configurations satisfy (R0) and (R1).

\begin{lemma} \label{lemma:pomocne}
	Let $l \in \N$. Then $\forall z \in U\left(0,\frac{1}{2}l\right)$ and $\forall y' \in  U(0,2l+1)^c$ the following inequalities hold
	\begin{equation*}
		\rho(z,(y',\abs{y'}-l)) > l^2 \geq \underset{w \in U(0,\frac{1}{2}l)}{\sup}\abs{w-z}^2.
	\end{equation*}
\end{lemma}
\begin{proof}
	Clearly the second inequality holds. For the first one, we can simply write
	\begin{align*}
		\rho(z,(y',&\abs{y'}-l)) = \abs{z-y'}^2 - (\abs{y'}-l)^2
		= \abs{z}^2 + \abs{y'}^2 -2\dotproduct{z}{y'} - \abs{y'}^2 + 2l\abs{y'} - l^2 \\
		& \geq \abs{z}^2 -2\abs{z}\abs{y'} + 2\abs{y'}l -l^2
		= \abs{z}^2 + l(\abs{y'} - l) + \abs{y'}(l-2\abs{z}) \geq l^2 + l > l^2.
	\end{align*}
\end{proof}

\begin{lemma}\label{lemma:podmr0r1protemp}
	It holds that all $\gamma \in \Msettemp$ satisfy (R0) and therefore the Laguerre cells $L(x,\gamma)$ are well defined. Furthermore, it holds that all $\gamma \in \Msettemp$ satisfy the first regularity condition (R1).
\end{lemma}
\begin{proof}
	Take $z \in \rtwo$ and $\gamma \in \Msett$, $t \in \N$. We want to show that there exists $\text{min}_{x \in \gamma } \rho(z,x)$. Clearly, if $\gamma \in \Msetf$, the assumption is satisfied. Consider an infinite configuration $\gamma$. We will use the property of tempered configurations given by (\ref{lemma:vlasttemp}), which states that there exists $l(t)$ such that $\forall l \geq l(t)$ the following implication holds
	\begin{equation} \label{prop1}
		(x',x'') \in \ms{\gamma}{(U(0,2l+1))^c} \implies B(x',x'') \cap U(0,l) = \emptyset.
	\end{equation}
	Choose $l$ large enough so that
	\begin{center}
		i) $l \geq l(t)$ \qquad
		and \qquad ii) $z \in U\left(0, \frac{1}{2}l\right)$ and there exists $x \in \gamma_{U\left(0, \frac{1}{2}l\right)}$.
	\end{center}
	Clearly, such $l$ can be chosen. Lemma \ref{lemma:pomocne} states that $\forall y' \in U(0,2l+1)^c$
	\begin{equation}\label{prop2}
		\rho(z, (y',\abs{y'} - l)) \geq \underset{w \in U(0,\frac{1}{2}l)}{\sup} \abs{w - z}^2.
	\end{equation}

	We know, because of property (\ref{prop1}), that
	\begin{equation*}
		y'' \leq \abs{y'} - l,\, 	\forall y = (y',y'') \in \ms{\gamma}{(U(0,2l+1))^c},
	\end{equation*}and therefore
	$\rho(z,y) = \abs{y'-z}^2 - (y'')^2  \geq \abs{y'-z}^2 - (\abs{y'} - l)^2 =  \rho(z,(y',\abs{y'} - l)).$
	
	Then, using (\ref{prop2}) together with point ii) above, we get that $\forall y \in \ms{\gamma}{(U(0,2l+1))^c}$  $$\rho(z,y) \geq \underset{w \in U(0,\frac{1}{2}l)}{\sup} \abs{w - z}^2 \geq \abs{x'-z}^2 \geq \rho(z,x)$$ and this completes the proof as then $\text{min}_{x \in \gamma } \rho(z,x) = \text{min}_{x \in \gamma_{U(0,2l+1)} } \rho(z,x),$
	which exists thanks to the local finiteness of $\gamma$.
	
	Now consider (R1). We want to show that for every $z \in \rtwo$ and $t \in \mathbb{R}$ only finitely many elements $y \in \gamma$ satisfy $\abs{z-y'}^2 - (y'')^2 \leq t $. But this is a clear consequence of the derivations above. Take $z \in \rtwo$ and $t \in \mathbb{R}$. Then there exists $l$ large enough such that $l^2 > t$ and such that it satisfies i) and ii). Then we have that $\forall y \in \ms{\gamma}{(U(0,2l+1))^c}$
	\begin{equation*}
		\abs{z-y'}^2 - (y'')^2 \geq \abs{z-y'}^2 - (\abs{y'}-l)^2 \geq  l^2 > t,
	\end{equation*}
	and therefore only the points $y \in \ms{\gamma}{U(0,2l+1)}$ (and there are finitely many of them) can satisfy $\abs{z-y'}^2 - (y'')^2 \leq t $.
\end{proof}

\section*{Acknowledgement}
\small
This work was partially supported by the Czech Science Foundation, project no.22-15763S. The author would also like to thank prof.\ Viktor Bene{\v s} for an introduction to the theory of Gibbs processes and for his useful remarks concerning this paper.



\end{document}